\documentclass[leqno,12pt]{article}
\usepackage{amsmath, amsfonts, amssymb,amsthm}

\usepackage{epsfig}
\usepackage{enumerate}
\usepackage{psfrag}
\usepackage{graphics}
\usepackage{epstopdf}
\usepackage{color}
\usepackage{pifont}
\usepackage[all]{xy}
\def\phi{\varphi}

\def\e{\varepsilon}
\def\eps{\epsilon}

\def\bi{\begin{itemize}}
\def\ei{\end{itemize}}

\newcommand{\oK}{\overline{K}}
\newcommand{\oP}{\overline{P}}

\newcommand{\ov}{\overline{v}}

\renewcommand{\qed}{{\flushright{\hfill \rule{2mm}{2mm}}}}

\newcommand{\fe}{{\mathfrak e}}

\newcommand{\wt}{\widetilde}
\newcommand{\wh}{\widehat}

\newcommand{\Ker}{{\mathrm {Ker}}}

\newcommand{\Coker}{\mathrm {Coker}}

\newcommand{\Int}{{\mathrm {Int\,}}}
\newcommand{\Id}{{\mathrm {Id}}}

\newcommand{\Z}{{\mathbb Z}}

\newcommand{\R}{\mathbb {R}}
\newcommand{\Q}{\mathbb {Q}\,}
\newcommand{\C}{\mathbb {C}\,}

\newcommand{\Fib}{\textsf{Fib}}
\newcommand{\FFib}{\textsf{FFib}}
\newcommand{\Fold}{\textsf{Fold}}

\newcommand{\p}{{\partial}}

\newcommand{\Ver}{\mathrm{Vert}}

\newcommand{\ppxx}{\frac{\partial}{\partial x_1}}
\newcommand{\ppxk}{\frac{\partial}{\partial x_k}}
\newcommand{\ppxkk}{\frac{\partial}{\partial x_{k+1}}}

\newcommand{\ppxdd}{\frac{\partial}{\partial x_{d+1}}}

\newcommand{\Tinf}{T_\infty}
\newcommand{\oTinf}{\overline T_\infty}
\newcommand{\oV}{\overline V}
\newcommand{\oC}{\overline C}

\newtheorem{theorem}{Theorem}[section]
\newtheorem{corollary}[theorem]{Corollary}
\newtheorem{construction}[theorem]{Construction}
\newtheorem{lemma}[theorem]{Lemma}
\newtheorem{proposition}[theorem]{Proposition}

\newtheorem{example}[theorem]{Example}

\newtheorem{remark}[theorem]{Remark}
\newtheorem{definition}[theorem]{Definition}
\def\proclaim #1. #2\par{\medbreak
  \noindent{\bf#1.\enspace}{\sl#2\par}%
   \ifdim\lastskip<\medskipamount \removelastskip\penalty55
        \medskip\fi}
\def\R{\mathbb{R}}
\def\bbR{\mathbb{R}}
\def\Z{\mathbb{Z}}

\def\cL{{\cal L}}

\def\oZ{\overline{Z}}

\def\Op{{\mathcal O}{\it p}\,}
\def\mn{\medskip\noindent}

\newcommand{\ff}{\mathfrak{f}}
\newcommand{\Emb}{\mathrm{Emb}}
\newcommand{\Diff}{\mathrm{Diff}}
\newcommand{\Cone}{\mathrm{Cone}}
\newcommand{\Span}{\mathrm{Span}}
\begin{document}
\title{Madsen-Weiss for geometrically minded topologists}

\author{Yakov Eliashberg
\footnote{Partially supported by NSF grants DMS-0707103 and DMS 0244663}
\\Stanford University
\\ USA
\and
S{\o}ren Galatius
\footnote{Partially supported by NSF grant DMS-0805843 and Clay Research Fellowship}
\\ Stanford University
\\USA
\and
Nikolai Mishachev
\footnote{Partially supported by  NSF grant DMS 0244663}
\\Lipetsk Technical University
\\Russia\\
 {\small  To D.B. Fuchs on his 70th birthday}}
\maketitle

\begin{abstract}
  We give an alternative proof of Madsen-Weiss' \emph{generalized
    Mumford conjecture}.  Our proof is based on ideas similar to
  Madsen-Weiss' original proof, but it is more geometrical and less
  homotopy theoretical in nature.  At the heart of the argument is a
  geometric version of \emph{Harer stability}, which we formulate as a
  theorem about folded maps.
\end{abstract}
\tableofcontents

\section{Introduction and statement of results}\label{intro}

Our main theorem gives a relation between \emph{fibrations} (or
\emph{surface bundles}) and a related notion of \emph{formal
  fibrations}.  By a fibration we shall mean a smooth map $f: M \to
X$, where $M$ and $X$ are smooth, oriented, compact manifolds and $f$
is a submersion (i.e.\ $df: TM \to f^* TX$ is surjective).  A
cobordism between two fibrations $f_0: M_0 \to X_0$ and $f_1: M_1 \to
X_1$ is a triple $(W,Y,F)$ where $W$ is a cobordism from $M_0$ to
$M_1$, $Y$ is a cobordism from $X_0$ to $X_1$, and $F: W \to Y$ is a
submersion which extends $f_0 \amalg f_1$.

\begin{definition}\label{defn:formal-fib}
  \begin{enumerate}[(i)]
  \item An \emph{unstable formal fibration} is a pair $(f,\phi)$,
    where $f: M \to X$ is a smooth proper map, and $\phi: TM \to f^*
    TX$ is a bundle epimorphism.
  \item A \emph{stable formal fibration} (henceforth just a formal
    fibration) is a pair $(f,\phi)$, where $f$ is as before, but
    $\phi$ is defined only as a \emph{stable} bundle map.  Thus for
    large enough $j$ there is given an epimorphism $\phi: TM \oplus
    \epsilon^j \to TX \oplus \epsilon^j$, and we identify $\phi$ with
    its stabilization $\phi \oplus \epsilon^1$.  A cobordism between
    formal fibrations $(f_0,\phi_0)$ and $(f_1, \phi_1)$ is a
    quadruple $(W,Y,F,\Phi)$ which restricts to $(f_0,\phi_0) \amalg
    (f_1, \phi_1)$.
  \item The formal fibration induced by a fibration $f: M \to X$ is
    the pair $(f, df)$, and a formal fibration is \emph{integrable} if
    it is of this form.
  \end{enumerate}
\end{definition}

Our main theorem relates the set of cobordism classes of fibrations 
with the set of cobordism classes of formal fibrations.  Let us first
discuss the \emph{stabilization} process (or more precisely
``stabilization with respect to genus''.  This should not be confused
with the use of ``stable'' in ``stable formal fibration''.  In the
former, ``stabilization'' refers to increasing genus; in the latter it
refers to increasing the dimension of vector bundles.)  Suppose $f: M
\to X$ is a formal fibration (we will often suppress the bundle
epimorphism $\phi$ from the notation) and $j : X \times D^2 \to M$ is
an embedding over $X$ (i.e. $f\circ j=\Id:X\to X$), such that $f$ is
integrable on the image of $j$.  Then we can \emph{stabilize} $f$ by
taking the fiberwise connected sum of $M$ with $X \times T$ (along
$j$), where $T = S^1 \times S^1$ is the torus.  If $f$ happens to be
integrable, this process increases the genus of each fiber by 1.


Our main theorem is the following.
\begin{theorem}\label{thm:main0}
  Let $f:M \to X$ be a formal fibration which is integrable over the
  image of a fiberwise embedding $j: X \times D^2 \to M$.  Then, after
  possibly stabilizing a finite number of times, $f$ is cobordant to
  an integrable fibration with connected fibers.
\end{theorem}
We will also prove a relative version of the theorem.  Namely, if $X$
has boundary and $f$ is already integrable, with connected fibers,
over a neighborhood of $\partial X$, then the cobordism in the theorem
can be also be assumed integrable over a neighborhood of the boundary.

Theorem~\ref{thm:main0} is equivalent to Madsen-Weiss' ``generalized
Mumford conjecture'' \cite{MW07} which states that a certain map
\begin{align}\label{eq:1}
  \Z \times B \Gamma_\infty \to \Omega^\infty \C P^\infty_{-1}
\end{align}
induces an isomorphism in integral homology.  In the rest of this
introduction we will explain the equivalence and introduce the methods
that go into our proof of Theorem~\ref{thm:main0}.  Our proof is
somewhat similar in ideas to the original proof, but quite different
in details and language.  The general scheme of reduction of some
algebro-topological problem to a problem of existence of bordisms
between formal and genuine (integrable) fibrations was first proposed
by D.\ B.\ Fuchs in \cite{Fu74}. That one might prove Madsen-Weiss'
theorem in the form of Theorem~\ref{thm:main0} was also suggested by
I.\ Madsen and U.\ Tillmann in \cite{MT01}.

\subsection{Diffeomorphism groups and mapping class groups}
\label{sec:diff-groups-mapp}

Let $F$ be a compact oriented surface, possibly with boundary
$\partial F = S$.  Let $\Diff(F)$ denote the topological group of
diffeomorphisms of $F$ which restrict to the identity on the boundary.
The classifying space $B\Diff(F)$ can be defined as the orbit space
\begin{align*}
  B\Diff(F) = \Emb(F,\R^\infty)/\Diff(F),
\end{align*}
and it is a classifying space for fibrations: for a manifold $X$,
there is a natural bijection between isomorphism classes of smooth
surface bundles $E \to X$ with fiber $F$ and trivialized boundary
$\partial E = X \times S$, and homotopy classes of maps $X \to
B\Diff(F)$.

The \emph{mapping class group} is defined as $\Gamma(F) = \pi_0
\Diff(F)$, i.e.\ the group of isotopy classes of diffeomorphisms of
the surface.  It is known that the identity component of $\Diff(F)$ is
contractible (as long as $g \geq 2$ or $\partial F \neq \varnothing$),
so $B\Diff(F)$ is also a classifying space for $\Gamma(F)$ (i.e.\ an
Eilenberg-Mac Lane space $K(\Gamma(F),1)$).  When $\partial F =
\varnothing$, this is also related to the \emph{moduli space} of Riemann
surfaces (i.e.\ the space of isomorphism classes of Riemann surfaces
of genus $g$) via a map
\begin{align}\label{eq:5}
  B \Diff(F) \to \mathcal{M}_g
\end{align}
which induces an isomorphism in rational homology and cohomology.
Mumford defined characteristic classes $\kappa_i \in
H^{2i}(\mathcal{M}_g;\Q)$ for $i\geq 1$ and conjectured that the
resulting map
\begin{align*}
  \Q[\kappa_1, \kappa_2, \dots] \to H^*(\mathcal{M}_g;\Q)
\end{align*}
is an isomorphism in degrees less than $(g-1)/2$.  This is the
original \emph{Mumford Conjecture}. 

It is convenient to take the limit $g \to \infty$.  Geometrically,
that can be interpreted as follows.  Pick a surface $F_{g,1}$ of genus
$g$ and with one boundary component.  Also pick an inclusion $F_{g,1}
\to F_{g+1,1}$.  Let $T_\infty$ be the union of the $F_{g,1}$ over all
$g$, i.e.\ a countably infinite connected sum of tori.  We will
consider fibrations $E \to X$ with trivialized ``$T_\infty$ ends''.

\begin{figure}[hi]
\centerline{\includegraphics[height=45mm]{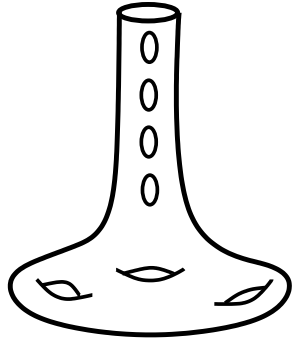}}
\caption{Surface with $\Tinf$-end}
\label{mw5}
\end{figure}

We are considering pairs $(f,j)$ where $f: E \to X$ is a smooth fiber
bundle with fiber $T_\infty$, and
\begin{align}\label{eq:6}
  j: X \times T_\infty \leadsto E
\end{align}
is a \emph{germ at infinity} of an embedding over $X$.  This means,
for $X$ compact, that representatives of $j$ are defined on the
complement of some compact set in $X \times T_\infty$, and their
images contain the complement of some compact set in $E$.

Let us describe a classifying space for fibrations with $T_\infty$
ends.  Let $B \Gamma_\infty$ be the the mapping telescope (alias
\emph{homotopy colimit}) of the direct system
\begin{align}\label{eq:4}
  B \Diff(F_{0,1}) \to B\Diff(F_{1,1}) \to B\Diff(F_{2,1}) \to \dots.
\end{align}
\begin{lemma}
  $\Z \times B\Gamma_\infty$ is a classifying space for fibrations
  with $T_\infty$ ends.
\end{lemma}
\begin{proof}
  Any compact $K \subset T_\infty$ is contained in $F_{g,1}
  \subset T_\infty$ for some finite $g$.  Let $T_{\infty}^g
  \subset T_\infty$ be the complement of $F_{g,1}$.  Let us consider
  for a moment fibrations with fiber $T_\infty$ and an embedding as in
  (\ref{eq:6}), but which is actually defined on $T_\infty^k$.  Call
  such bundles \emph{$k$-trivialized}.  Specifying a $k$-trivialized
  bundle is the same thing as specifying a fibration $E' \to X$ with
  connected, compact fibers, and trivialized boundary $\partial E' = X
  \times S^1$ (namely $E'$ is the complement of the image of $j$).
  Thus, the disjoint union
  \begin{align}\label{eq:7}
    B = \coprod_g B \Diff(F_{g,1})
  \end{align}
  is a classifying space for $k$-trivialized bundles (notice $B$ is
  independent of $k$).  A $k$-trivialized bundle is also a
  $(k+1)$-trivialized bundle, and increasing $k$ is represented by a
  ``stabilization'' self-map $s: B \to B$.  In the
  representation~\eqref{eq:7}, $s$ maps $B\Diff(F_{g,1})$ to
  $B\Diff(F_{g+1,1})$ and this is induced by the same map as
  in~\eqref{eq:4}.

  Now the statement follows
  by taking the direct limit as $k \to \infty$, and noticing that $\Z
  \times B\Gamma_\infty$ is the homotopy colimit of the direct system
  \begin{align*}
    B \xrightarrow{s} B \xrightarrow{s} \dots.
  \end{align*}
\end{proof}
 
\subsection{A Thom spectrum}\label{sec:thom-spectrum}
In this section we define a space $\Omega^\infty \C P_{-1}^\infty$ and
interpret it as a classifying space for formal fibrations.  The
forgetful functor from fibrations to formal fibrations is represented
by a map
\begin{align}\label{eq:17}
  B\Diff(F) \to \Omega^\infty \C P_{-1}^\infty.
\end{align}
We then consider the same situation, but where fibrations and formal
fibrations have $T_\infty$ ends.  This changes the source of the
map~\eqref{eq:17} to $\Z \times B\Gamma_\infty$, but turns out to not
change the homotopy type of the target.  We get a map
\begin{align*}
  \Z \times B\Gamma_\infty \to \Omega^\infty \C P_{-1}^\infty,
\end{align*}
representing the forgetful functor from fibrations with $T_\infty$
ends to formal fibrations with $T_\infty$ ends.

The space $\Omega^\infty \C P_{-1}^\infty$ is defined as the Thom
spectrum of the negative of the canonical complex line bundle over $\C
P^\infty$.  In more detail, let $\mathrm{Gr}_2^+(\R^N)$ be the
Grassmannian manifold of oriented 2-planes in $\R^N$.  It supports a
canonical 2-dimensional vector bundle $\gamma_{N}$ with an
$(N-2)$-dimensional complement $\gamma_{N}^\perp$ such that $\gamma_N
\oplus \gamma_{N}^\perp = \epsilon^{N}$.  There is a canonical
identification
\begin{align}\label{eq:2}
  \gamma_{N+1}^\perp | \mathrm{Gr}_2^+(\R^N) = \gamma_N^\perp \oplus
  \epsilon^1.
\end{align}
The Thom space $\mathrm{Th}(\gamma_N^\perp)$ is defined as the
one-point compactification of the total space of $\gamma_N^\perp$, and
the identification~(\ref{eq:2}) induces a map $S^1 \wedge
\mathrm{Th}(\gamma_N) \to \mathrm{Th}(\gamma_{N+1})$.  The space
$\Omega^\infty \C P^\infty_{-1}$ is defined as the direct limit
\begin{align*}
  \Omega^\infty \C P^\infty_{-1} = \lim_{N \to \infty} \Omega^N
  \mathrm{Th}(\gamma_N^\perp).
\end{align*}

Like we did for $B\Diff(F)$, we shall think of $\Omega^\infty \C
P^\infty_{-1}$ as a classifying space, i.e.\ interpret homotopy
classes of maps $X \to \Omega^\infty \C P^\infty_{-1}$ from a smooth
manifold $X$ in terms of certain geometric objects over $X$.  Recall
the notion of \emph{formal fibration} from
Definition~\ref{defn:formal-fib} above.  A cobordism $(W,Y,F,\Phi)$ of
formal fibrations is a \emph{concordance} if the target cobordism is a
cylinder: $Y = X \times [0,1]$.
\begin{lemma}\label{lemma:PT1}
  There is a natural bijection between set
  \begin{align*}
    [X, \Omega^\infty \C P^\infty_{-1}]
  \end{align*}
  of homotopy classes of maps, and the set of concordance classes of
  formal fibrations over $X$.
\end{lemma}
\begin{proof}[Proof sketch]
  This is the standard argument of Pontryagin-Thom theory.  In one
  direction, given a map $X \to \Omega^N \mathrm{Th}(\gamma_N^\perp)$,
  one makes the adjoint map $g: X \times \R^N \to
  \mathrm{Th}(\gamma_N^\perp)$ transverse to the zero section of
  $\gamma_N^\perp$ and sets $M = g^{-1}(\text{zero-section})$.  Then
  $M$ comes with a map $c:M \to \mathrm{Gr}_2^+(\R^N)$ and the normal
  bundle of $M \subset X \times \R^N$ is $c^*(\gamma_N^\perp)$.  This
  gives a stable isomorphism $TM \cong_{\mathrm{st}} TX \oplus
  c^*(\gamma_N)$ and hence a stable epimorphism $TM \to TX$.

  In the other direction, given a formal fibration $(f,\phi)$ with $f:
  M \to X$, we pick an embedding $M \subset X \times \R^N$.  Letting
  $\nu$ be the normal bundle of this embedding, we get a ``collapse''
  map
  \begin{align}
    \label{eq:10}
    X_+ \wedge S^N \to \mathrm{Th}(\nu).
  \end{align}
  We also get an isomorphism of vector bundles over $M$
  \begin{align}\label{eq:8}
    TM\oplus \nu \cong f^* TX \oplus \epsilon^N.
  \end{align}
  Let $\xi: M \to \mathrm{Gr}_2^+(\R^N)$ be a classifying map for the
  kernel of the stable epimorphism $\phi: TM \to TX$ (this is a
  two-dimensional vector bundle with orientation induced by the
  orientations of $X$ and $M$), so we have a stable isomorphism $TM
  \cong_{\mathrm{st}} TX \oplus \xi^* \gamma_N$.  Combining
  with~(\ref{eq:8}) we get a stable isomorphism $\xi^* \gamma_N \oplus
  \nu \cong_{\mathrm{st}} \epsilon^N$.  By adding
  $\xi^*(\gamma_N^\perp)$ we get a stable isomorphism $\nu \cong_{st}
  \xi^*(\gamma_N^\perp) \oplus \epsilon^N$ which we can assume is
  induced by an unstable isomorphism (since we can assume $N \gg \dim
  M$)
  \begin{align}
    \label{eq:9}
    \nu \cong \xi^* \gamma^\perp_N.
  \end{align}
  This gives a proper map $\nu \to \gamma_N^\perp$ and hence a map of
  Thom spaces $\mathrm{Th}(\nu) \to \mathrm{Th}(\gamma^\perp_N)$.
  Compose with (\ref{eq:10}) and take the adjoint to get a map $X \to
  \Omega^N \mathrm{Th}(\gamma^\perp_N)$.  Finally let $N \to \infty$
  to get a map $X \to \Omega^\infty \C P^\infty_{-1}$.

  The homotopy class of the resulting map $X \to \Omega^\infty \C
  P^\infty_{-1}$ is well defined and depends only on the concordance
  class of the formal fibration $f: M \to X$.
\end{proof}

A fibration naturally gives rise to a formal fibration, and this
association gives rise to a map of classifying spaces which is the
map~(\ref{eq:5}).  We would like to make this process compatible with
the stabilization procedure explained in
section~\ref{sec:diff-groups-mapp} above.  To this end we consider
formal fibrations with $k$-trivialized $T_\infty$ ends.  This means
that $f: M \to X$ is equipped with an embedding over $X$
\begin{align*}
  j: X \times T_{\infty}^k \to M,
\end{align*}
and that $(f,\phi)$ is integrable on the image of $j$.  Of course, we
also replace the requirement that $M$ be compact by the requirement
that the complement of the image of $j$ be compact.
\begin{lemma}\label{lemma:PT2}
  Formal fibrations with $k$-trivialized ends are represented by the
  space $\Omega^\infty \C P^\infty_{-1}$.
\end{lemma}
\begin{proof}[Proof sketch]
  This is similar to the proof of Lemma~\ref{lemma:PT1} above.
  Applying the Pontryagin-Thom construction from the proof of that
  lemma to the projection $X \times T_\infty^k \to X$ gives a path
  $\alpha_0: [k,\infty) \to \Omega^{N-1} \mathrm{Th}(\gamma_N^\perp)$.
  Applying the Pontryagin-Thom construction to an arbitrary
  $k$-trivialized formal fibration gives a path $\alpha: [0,\infty)
  \to \Omega^{N-1} \mathrm{Th}(\gamma_N^\perp)$ whose restriction to
  $[k,\infty)$ is $\alpha_0$.  The space of all such paths is homotopy
  equivalent to the loop space $\Omega^{N}
  \mathrm{Th}(\gamma_N^\perp)$.
\end{proof}
Increasing $k$ gives a diagram of stabilization maps
\begin{align*}
  \xymatrix{ {\coprod_g B \Diff(F_{g,1})} \ar[d]^s \ar[r] &
    {\Omega^\infty \C
      P^\infty_{-1}} \ar[d]^s\\
    {\coprod_g B \Diff(F_{g,1})} \ar[r] & {\Omega^\infty \C
      P^\infty_{-1}.}
  }
\end{align*}
On the right hand ``formal'' side, the stabilization map is up to
homotopy described as multiplication by a fixed element
(multiplication in the loop space structure.  See the proof of
Lemma~\ref{lemma:PT2}.)  In particular it is a homotopy equivalence,
so the direct limit has the same homotopy type ${\Omega^\infty \C
  P^\infty_{-1}}$.  Taking the direct limit we get the desired
map
\begin{align*}
  \Z \times B \Gamma_\infty \to {\Omega^\infty \C P^\infty_{-1}}.
\end{align*}
This is the map~\eqref{eq:1}.  The target of this map should be
thought of as the homotopy direct limit of the system
\begin{align*}
  {\Omega^\infty \C P^\infty_{-1}} \xrightarrow{s} {\Omega^\infty \C
    P^\infty_{-1}} \xrightarrow{s} \dots
\end{align*}
and as a classifying space for formal fibrations with $T_\infty$ ends.
(In some sense it is a ``coincidence'' that the classifying space for
formal fibrations and the classifying space for formal fibrations with
$T_\infty$ ends have the same homotopy type).


\subsection{Oriented bordism}

For a pair $(X,A)$ of spaces, oriented bordism $\Omega_n(X,A) =
\Omega^{\mathrm{SO}}_n(X,A)$ is defined as the set of bordism classes
of continuous maps of pairs
\begin{equation*}
  f: (M, \partial M) \to (X,A)
\end{equation*}
for smooth oriented compact $n$-manifolds $M$ with boundary $\partial
M$.   To be precise, a bordism between two maps
$f_\pm: (M_\pm, \partial M_\pm) \to (X,A)$ is a map $F:(W,\p'W)\to (X,A)$, where $W$ is a compact, oriented manifold   with boundary with 
corners, so that $\p W=\p_-W\cup\p'W\cup\p_+W$, where $\p_\pm W=M_\pm$ and  $\p'W$ is a cobordism between closed manifolds $\p M_-$ and $\p M_+$, and the map
$F: (W,\partial' W) \to (X,A)$ such that $F|_{\p_\pm W}=f_\pm$.   

For a single space $X$ set $\Omega_n(X) = \Omega_n(X,\varnothing)$.
Oriented bordism is a \emph{generalized homology theory}.  This means
that it satisfies the usual Eilenberg-Steenrod axioms for homology
(long exact sequence etc) except for the dimension axiom.  In
particular a map $A \to X$ induces an isomorphism $\Omega_*(A) \to
\Omega_*(X)$ if and only if the relative groups $\Omega_*(X,A)$ all
vanish.  The following result is well known.  It follows easily from
the Atiyah-Hirzebruch spectral sequence (for completeness we give a
geometric proof in Appendix B).
\begin{lemma}\label{lem:bordism-to-homology}
  Let $f: X \to Y$ be a continuous map of topological spaces.  Then
  the following statements are equivalent.
  \begin{enumerate}[(i)]
  \item $f_*: H_k(X) \to H_k(Y)$ is an isomorphism for $k < n$ and an
    epimorphism for $k = n$.
  \item $f_*: \Omega_k(X) \to \Omega_k(Y)$ is an isomorphism for $k <
    n$ and an epimorphism for $k = n$.
  \end{enumerate}
  In particular, $f$ induces an isomorphism in homology in all degrees
  if and only if it does so in oriented bordism.
\end{lemma}

We apply this to the pair $(\Omega^\infty \C P_{-1}^\infty, \Z \times
B\Gamma_\infty)$.  Interpreting $\Z \times B\Gamma_\infty$ and
$\Omega^\infty \C P_{-1}^\infty$ as classifying spaces for fibrations,
resp.\ formal fibrations, with $T_\infty$ ends, we get the following
interpretaion.
\begin{lemma}\label{lemma:interpret-relative-bord}
  There is a natural bijection between the relative oriented bordism
  groups $\Omega_*(\Omega^\infty \C P_{-1}^\infty, \Z \times
  B\Gamma_\infty)$ and cobordism classes of formal fibrations $f: M
  \to X$ with $T_\infty$ ends.  The formal fibration is required to be
  integrable over a neighborhood of $\partial X$, and cobordisms $F: W
  \to Y$ are required to be integrable over a neighborhood of
  $\partial' Y$.
\end{lemma}

That~\eqref{eq:1} induces an isomorphism in integral homology
(Madsen-Weiss' theorem) is now, by
Lemma~\ref{lem:bordism-to-homology}, equivalent to the statement that
the relative groups
\begin{align*}
  \Omega_*(\Omega^\infty \C P_{-1}^\infty, \Z \times B\Gamma_\infty)
\end{align*}
all vanish.  By Lemma~\ref{lemma:interpret-relative-bord}, this is
equivalent to
\begin{theorem}
  \label{thm:main1}
  Any formal fibration $f: M \to X$ with $T_\infty$ ends is cobordant
  to an integrable one.  If $(f,\phi)$ is already integrable over
  $\partial X$, then the cobordism can be assumed integrable over
  $\partial'$.
\end{theorem}

Theorem~\ref{thm:main1} is our main result.  It is a geometric
version of Madsen-Weiss' theorem.  It is obviously equivalent to
Theorem~\ref{thm:main0} above (with its relative form).

\subsection{Harer stability}\label{sec:Harer-intro}

J.\ Harer proved a homological stability theorem in \cite{Ha85} which
implies precise bounds on the number of stabilizations needed in
Theorem~\ref{thm:main0}.  At the same time, it will be an important
part of the proof of the same theorem (as it does in \cite{MW07}).

Roughly it says that the homology of the mapping class group of a
surface $F$ is independent of the topological type $F$, as long as the
genus is high enough.  The result was later improved by Ivanov
(\cite{Iv89,Iv93}) and then by Boldsen \cite{Bo09}. We state the precise result.

Consider an inclusion $F \to F'$ of compact, connected,
oriented surfaces.  Let $S = \partial F'$, and let $\Sigma\subset
F'$ denote the complement of $F$.  Thus $F' = F \cup_{\partial F}
\Sigma$.  There is an induced map of classifying spaces
\begin{align}\label{eq:3}
  B\Diff(F) \to B\Diff(F').
\end{align}
A map $f:X \to B\Diff(F')$ classifies a fibration $E \to X$ with
fiber $F'$ and boundary $\partial E = X \times S$, where $S = \partial
F'$.  Lifting it to a map into $B\Diff(F)$ amounts to extending the
embedding $X \times S \to E$ to an embedding
\begin{align*}
  X \times \Sigma \to E
\end{align*}
over $X$.

The most general form of Harer stability states that the
map~(\ref{eq:3}) induces an isomorphism in $H_k(-;\Z)$ for $k <
2(g-1)/3$, where $g$ is the genus of $F$.  Consequently, by
Lemma~\ref{lem:bordism-to-homology}, it induces an isomorphism in
oriented bordism $\Omega_n(-)$ for $n < 2(g-1)/3$ or, equivalently, the
relative bordism group
\begin{align*}
  \Omega_n(B\Diff(F'), B\Diff(F))
\end{align*}
vanishes for $n < 2(g-1)/3$.  Thus, Harer stability has the following
very geometric interpretation: For any fibration $f:E \to X$ with
fiber $F'$ and boundary $\partial E = X \times S$, $f$ is cobordant to
a fibration $f': E' \to X'$ via a cobordism $F: W \to M$ (which is a
fibration with trivialized boundary $M \times S$, which restricts to
$f \amalg f'$) where the embedding $X' \times S \to E'$ extends to an
embedding
\begin{align*}
  X' \times \Sigma \to E'
\end{align*}
Moreover, this can be assumed compatible with any given extension
$(\partial X) \times \Sigma \to E$ over the boundary of $X$.  Here we
assume $F' = F\cup_{\partial F} \Sigma$ as above, that $F$ and $F'$
are connected, and that $F$ has large genus.  If the fibration has
$T_\infty$ ends, the genus assumption is automatically satisfied, and
we get the following corollary.

\begin{theorem}[Geometric form of Harer stability]
  \label{thm:Harer-geometric}
  Let $\Sigma_1 \subset \Sigma_2$ be compact surfaces with boundary
  (not necessarily connected).  Let $f: M \to X$ be a fibration with
  $T_\infty$ ends, and let
  \begin{align*}
    j: (\partial X \times \Sigma_2) \cup (X \times \Sigma_1) \to M
  \end{align*}
  be a fiberwise embedding over $X$, such that in each fiber the
  complement of its image is connected.
 \begin{figure}[hi]
  \centerline{\includegraphics[height=60mm]{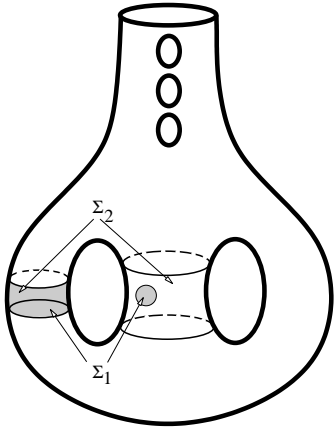}}
  \caption{ $\Sigma_1\subset \Sigma_2\subset F$}
  \label{mw15}
\end{figure}
  Then, after possibly changing $f: M \to X$ by a bordism which is the
  trivial bordism over $\partial X$, the embedding $j$ extends to an
  embedding of $X \times \Sigma_2$.
\end{theorem}
Explicitly, the bordism in the theorem is a fibration $F: W \to Y$
with $T_\infty$ ends, where the boundary $\partial Y$ is partitioned
as $\partial Y = X \cup X'$ with $\partial X = \partial X' = X \cap X'$
and $F_{|X} = f$.  The extension of $j$ is a fiberwise embedding
\begin{align*}
  J: (Y\times \Sigma_1) \cup (X' \times \Sigma_2) \to W
\end{align*}
over $Y$.

\subsection{Outline of proof}

\subsubsection{From formal fibrations to folded maps}\label{sec:formal-folded}
Given a formal fibration $(f:M\to X,\phi)$ with $\Tinf$ ends, the
overall aim is to get rid of all singularities of $f$ after changing
it via \emph{bordisms}.  Our first task will be to simplify the
singularities of $f$ as much as possible using only \emph{homotopies}.
The simplest generic singularities of a map $f:M\to X$ are {\it
  folds}. The fold locus $\Sigma(f)$ consists of points where the rank
of $f$ is equal to $\dim X-1$, while the restriction
$f|_{\Sigma(f)}:\Sigma(f)\to M$ is an immersion.  In the case when
$\dim M=\dim X+2=n+2$, which is the case we consider in this paper, we
have $\dim\Sigma(f)=n-1$.  A certain additional structure on folded
maps, called an {\it enrichment}, allows one to define a homotopically
canonical \emph{suspension}, i.e.\ a bundle epimorphism $\phi=\phi_f:
TM \oplus \epsilon^1 \to TX \oplus \epsilon^1$, such that $(f,\phi)$
is a formal fibration.  The enrichment of a folded map $f$ consists of
\begin{itemize}
\item an $n$-dimensional submanifold $V\subset M$ such that $\p
  V=\Sigma(f)$, and the restriction of $f$ to each connected component
  of $V_i\subset V$ is an embedding $\Int V_i\to X$;
\item a trivialization of the bundle $\Ker df|_{\Int V}$ with a
  certain additional condition on the behavior of this trivialization
  on $\p V=\Sigma(f)$.
\end{itemize}
 
Of course, existence of an enrichment is a strong additional condition
on the fold map.  In Section~\ref{sec:enriched}, we explain how to
associate to an enriched folded map $(f,\e)$ a formal fibration
$(f,\mathcal{L}(f,\e))$, where $\mathcal{L}(f,\e): TM \oplus
\epsilon^1 \to TX\oplus \epsilon^1$ is a bundle epimorphism associated
to the enrichment $\epsilon$.  The main result of
Section~\ref{sec:prelim} is Theorem~\ref{thm:epi-folds}, which proves
that any formal fibration can be represented in this way (plus a
corresponding relative statement).  This is proved using the
$h$-principle type result proven in \cite{EM97}.  Note that Theorem~\ref{thm:epi-folds} is a variation of the main result from \cite{El72} and can also be proven by the methods of that paper.
 Also in Section
\ref{sec:prelim} we recall some basic facts about folds and other
simple singularities of smooth maps, and discuss certain surgery
constructions needed for the rest of the proof of Theorem
\ref{thm:main1}.  This part works independently of the codimension
$d=\dim N-\dim M$, and hence the exposition in this section is done
for arbitrary $d>0$.

\subsubsection{Getting rid of elliptic folds}\label{sec:elliptic}
Each fold component has an index which is well defined provided that
the projection of the fold is co-oriented.  Assuming this is done,
folds in the case $d=2$ can be of index $0,1,2$ and $3$. We call folds
of index $1,2$ {\it hyperbolic} and folds of index $0,3$ {\it
  elliptic}. It is generally impossible to get rid of elliptic folds
by a homotopy of the map $f$. However, it is easy to do so  if one allows to change $f$
to a {\it bordant} map $\wt f:\wt M\to X$.  This bordism trades each
elliptic fold component by a parallel copy of a hyperbolic fold, see
Figure \ref{mw11} and Section \ref{sec:to-hyperbolic} below. A similar
argument allows one to make all fibers $\wt f^{-1}(x),x\in X$,
connected (comp. 
  \cite{MW07}).

\subsubsection{Generalized Harer stability theorem}\label{sec:gen-Harer}

A generalization of Harer's stability theorem to enriched folded maps,
see Theorem \ref{thm:Harer2} below, plays an important role in our
proof.  In Section~\ref{sec:proof-Harer} below we deduce Theorem
\ref{thm:Harer2} from Harer's Theorem \ref{thm:Harer-geometric} by
induction over strata in the stratification of the image
$f(\Sigma)\subset X$ of the fold, according to multiplicity of
self-intersections. 

\subsubsection{Getting rid of hyperbolic folds}\label{sec:final-step}
Let $f$ be an enriched folded map with hyperbolic folds and with
connected fibers.
Let $C$ be one of the fold components and $\oC=f(C)\subset X$ its
image. For the purpose of this introduction we will consider only the following special case.
First, we will assume that $C$ is homologically trivial. As we will see,   when $\dim X>1$ this will always be possible to arrange. In particular, $\oC$ bounds a domain $U_C\subset X$.
Next, we will assume that  the fold $C$ has index 1 with respect to the outward coorientation of the boundary of the domain $U_C$.   In other words, when the point $x\in
X$ travels across $\oC$ inside $U_C$ then
one of the circle in the fiber $f^{-1}(x)$ collapses to a point, so
locally the fiber gets disconnected to two discs, see Figure
\ref{mw13}. The inverse index 1 surgery makes a connected sum of two
discs at their centers.
 \begin{figure}[hi]
  \centerline{\includegraphics[height=50mm]{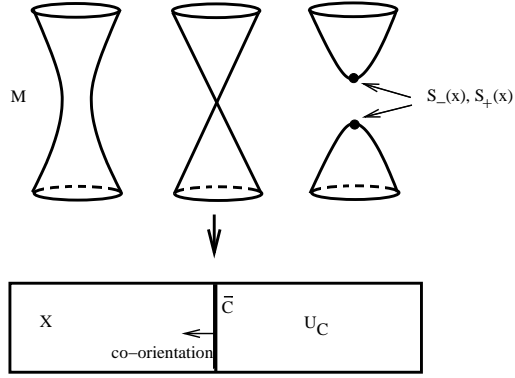}}
  \caption{Fibers of an index 2 fold}
  \label{mw13}
\end{figure}
  Note that on an open collar $\Omega=\p
U_C\times(0,1)\subset \Int U_C$ along $C$ in $U_C$ there exists two sections
$S_\pm:\Omega\to M$ such that the $0$-sphere $\{S_-(x), S_+(x)\}$ is
the ``vanishing cycle'', for the index $1$ surgery when $x$ travels
across $\oC$. Moreover, the enrichment structure ensures that the
vertical bundle along these sections is trivial. If one could extend
the sections $S_\pm$ to all of $U_C$ preserving all these properties,
then the fiberwise index $1$ surgery, attaching $1$-handle along small
discs surrounding $S_\pm(x)$ and $S_\pm(x)$, $x\in U_C$, would eliminate
the fold $\oC$. This is one of the fold surgeries described in detail
in Section \ref{sec:surgery-general}.

Though such extensions $S_\pm(x):\Int U_C\to M$ need not exist for our
original folded map $f$, Harer stability theorem in the form
\ref{sec:surgery-general} states that there is an enriched folded map
$\wt f:\wt M\to\wt X$, bordant to $f$, for which such sections do
exist, and hence the fold $\oC$ could be eliminated.  
   
\subsubsection{Organization of the paper}
As already mentioned, Section \ref{sec:prelim} recalls basic
definition and necessary results and constructions involving fold
singularities. In Section \ref{sec:folds-cusps} we define folded maps.
The goal of Section \ref{sec:double-folds} is Theorem
\ref{thm:special-folded} which is an $h$-principle for constructing
so-called {\it special folded} maps, whose folds are organized in pairs
of spheres.  This theorem is a reformulation of the Wrinkling Theorem
from \cite{EM97}.  We deduce Theorem~\ref{thm:special-folded} from the
Wrinkling Theorem in Appendix A.  In Section \ref{sec:enriched} we
define the notion of enrichment for folded maps and prove that an
enriched folded map admits a homotopically canonical suspension and
hence gives rise to a formal fibration.  The rest of
Section~\ref{sec:prelim} will prove that any formal fibration is
cobordant to one induced by an enriched folded map.  Section
\ref{sec:surgery-general} is devoted to fold surgery constructions
which we use later in the proof of the main theorem. These are just
fiberwise Morse surgeries, in the spirit of surgery of  singularities
techniques developed in \cite{El72}.  For further applications we need a  version of Theorem \ref{thm:special-folded}
 applicable to a slightly stronger version of formal fibrations
$(f,\phi)$ when $\phi:TM\to TX$ is a surjective map between the {\it
  non-stabilized} tangent bundle.  In Section \ref{sec:destab} we
explain the modifications which are necessary in the stable case, and
in Section \ref{sec:epi-folds} we formulate and prove
Theorem \ref{thm:epi-folds}, reducing formal fibrations to enriched folded
maps.
   
In Section \ref{sec:cob} we introduce several special bordism
categories and formulate the two remaining steps of the
proof: Proposition \ref{prop:enriched-hyperbolic} which allows us to
get rid of elliptic folds, and Proposition
\ref{prop:hyperbolic-fibrations} which eliminates the remaining
hyperbolic folds. 
 
Section \ref{sec:Harer-generalized} is devoted to the proof of the Harer
stability theorem for folded maps (Theorem \ref{thm:Harer2}).  We
conclude the proof of  Theorem \ref{thm:main1} in Section
\ref{sec:main-proof} by proving Proposition
\ref{prop:enriched-hyperbolic} in \ref{sec:to-hyperbolic} and
Proposition \ref{prop:hyperbolic-fibrations} in
\ref{sec:to-fibrations}.  In Section \ref{sec:misc} we collect two
Appendices. In Appendix A we deduce Theorem \ref{thm:special-folded}
from the Wrinkling Theorem from \cite{EM97}. Appendix B is devoted to
a geometric proof of Lemma~\ref{lem:bordism-to-homology}.

\section{Folded maps}\label{sec:prelim}

\subsection{Folds}\label{sec:folds-cusps}

Let $M$ and $X$ be smooth manifolds of dimension $m=n+d$ and $n$,
respectively.\footnote{For applications in this paper we will need
  only the case $d=2$.}  For a smooth map $f:M \rightarrow X$ we will
denote by $\Sigma(f)$ the set of its singular points, i.e.
$$
\Sigma(f) = \left\{ p \in M, \; \hbox { rank } d_pf < n \right\}\;.
$$
A point $p \in \Sigma(f)$ is called a {\it fold} type singularity or a
{\it fold} of index $k$ if near the point $p$ the map $f$ is
equivalent to the map
$$
\bbR ^{n-1} \times \bbR ^{d+1} \rightarrow \bbR ^{n-1} \times \bbR
^{1}
$$
given by the formula
\begin{equation}\label{eq:fold}
(y,x) \mapsto \left( y,\,\, Q(x)= -\sum^k_1 x^2_i + \sum_{k+1}^{d+1} x^2_j
\right)
\end{equation}
where $x=(x_1, \dots, x_{d+1}) \in \bbR ^{d+1}$ and
$y=(y_1,...,y_{n-1}) \in \bbR ^{n-1}$. We will also denote
$x_-=(x_1,\dots, x_k),\; x_+=(k_{k+1},\dots, x_{d+1})$ and write
$Q(x)=-|x_-|^2+|x_+|^2$.  For $M=\bbR ^1$ this is just a
non-degenerate index $k$ critical point of the function $f:V
\rightarrow \bbR ^1$. By a {\it folded map} we will mean a map with
only fold type singularities.
Given $y\in\bbR^{n-1}$ and  an  $\epsilon>0$ we will call a $(k-1)$-dimensional sphere $y\times\{Q(x)=-\epsilon, x_+=0\}\subset\bbR^{n-1}\times\bbR^{d+1}$ the {\it vanishing cycle} of the fold over the point
$(y,-\epsilon)$.

\mn The normal bundle of the image $\oC = f(C)$ of the fold is a real
line bundle over $C$.  A \emph{coorientation} of $\oC$ is a
trivialization of this line bundle.  A choice of coorientation allows one to provide each fold component $C$ with
 a well-defined index $s$, which changes from $s$ to $d+1-s$
with a switch of the coorientation.  The normal bundle of $C$ is $\Ker\,
df$, and the second derivatives of $f$ gives an invariantly defined non-degenerate
quadratic form $d^2 f:\Ker\, df|_C\to\Coker\,\,df$.    Denote $\Cone_\pm(C):=\{z\in\Ker\,df; \pm d^2f(z)>0\}$.    
There is a splitting  $$\Ker\, df|_{C}= \Ker_-(C)\oplus\Ker_+(C), $$ which is defined uniquely up to homotopy by the condition
$\Ker_\pm(C)\setminus 0\subset\Cone_\pm(C)$.
 



  


\subsection{Double folds and special folded
  mappings}\label{sec:double-folds}

Given an orientable $(n-1)$-dimensional manifold $C$, let us consider
the map
\begin{align}
  \label{eq:double-fold}
  w_C(n+d,n,k):C\times \bbR^1 \times \bbR^{d}
  \rightarrow C \times \bbR^1
\end{align}
given by the formula
\begin{align}\label{eq:double-fold2}
  (y,z,x) \mapsto
  \left( y, z^3 - 3z - \sum^k_1 x^2_i + \sum_{k+1}^{d} x_j^2 \right)\,,
\end{align}
where $y\in C,\,\, z \in \bbR^1$ and   $x\in \bbR ^{d}$.

\mn The singularity $\Sigma (w_C(n+d,n,k))$ consists of two copies of
the manifold $C$:
$$
C\times S^0\times 0\subset C\times \bbR^1\times\bbR ^{d}.
$$
The manifold $C\times 1$ is a fold of index $k$, while $C\times \{-1
\}$ is a fold of index $k+1$, with respect to the coorientation of
the folds in the image given by the vector field $\frac\p{\p u}$ where
$u$ is the coordinate $C\times\R\to\R$. It is important to notice that
the restriction of the map $w_C(n+d,n,k)$ to the annulus
$$A=C\times \Int D^1=C\times \Int D^1\times 0
\subset C\times \bbR^1\times\bbR ^{d}$$
is an {\it embedding}.

\begin{figure}[hi]
  \centerline{\includegraphics[height=40mm]{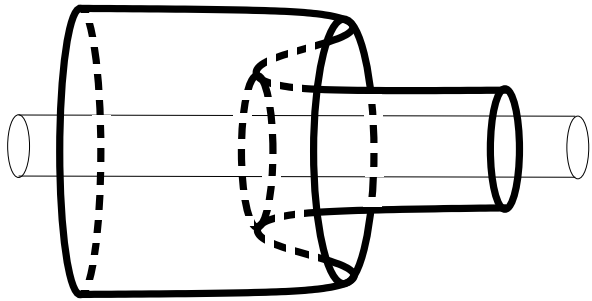}}
  \caption{The radial projection to the cylinder has a double fold along $C=S^1$}
  \label{mw6}
\end{figure}
\mn
Although the differential
$$dw_C(n+d,n,k):T(C\times\bbR^1\times\bbR^{d})
\rightarrow T(C\times\bbR^1)$$ degenerates at points of singularity
$\Sigma(w_C)$, it can be canonically {\it regularized} over $\Op
(C\times D^1)$, an open neighborhood of the annulus $C\times D^1$.
Namely, we can change the element $3(z^2-1)$ in the Jacobi matrix
of $w_C(n+d,n,k)$ to a positive function $\gamma$, which coincides
with $3(z^2-1)$ on $\bbR^1\setminus [-1-\delta,1+\delta]$ for
sufficiently small $\delta$.  The new bundle map
$${\cal R}(dw_C):T(C\times \bbR^1\times \bbR ^{d})
\rightarrow T(C\times \bbR^1)$$ provides a homotopically canonical
extension of the map
$$dw_C:T(C\times \bbR^1\times \bbR^{d}\setminus\Op (C\times D^1))
\rightarrow T(C\times \bbR ^1)$$ to an epimorphism (fiberwise
surjective bundle map)
$$T(C\times \bbR^1\times \bbR ^{d})
\rightarrow T(C\times \bbR^1)$$ We call ${\cal R}(dw_C)$ the {\it
  regularized differential} of the map $w_C(n+d,n,k)$.

\mn A map $f:U\rightarrow X$ defined on an open 
$U \subset M$ is called a {\it double $C$-fold} of index $k +\frac {1
}{ 2}$\,\, if it is equivalent to the restriction of the map
$w_C(n+d,n,k)$ to $\Op(C\times D^1)$. For instance, when $X=\R$ and
$C$ is a point, a double $C$-fold is a Morse function given in a
neighborhood of a gradient trajectory connecting two critical points
of neighboring indices.  In the case of general $n$, a double $C$-fold is called a
{\it spherical double fold} if $C=S^{n-1}$.

It is always easy to create a double $C$-fold as the following lemma
shows. This lemma is a parametric version of the creation of two new
critical points of a Morse function.

\begin{lemma}\label{lm:2fold-creation}
  Given a submersion $f:U\to X$ of a manifold $U$, a closed
  submanifold $C\subset U$ of dimension $n-1$ such that $f|_C:C\to X$
  is an embedding with trivialized normal bundle, and a splitting
  $K_-\oplus K_+$ of the vertical bundle $\Ver= \Ker\,df$
  over $\Op C$, one can
  construct a map $\wt f:U\to X$ such that
  \begin{itemize}
  \item $\wt f$ coincides with $f$ near $\p U$;
  \item in a neighborhood of $C$ the map $\wt f$ has a double $C$-fold,
   i.e. it is equivalent to the map \eqref{eq:double-fold2}
    restricted to $\Op(A=C\times D^1)$, where the
    frames $$\left(\ppxx,\dots,
      \ppxk\right)\;\;\hbox{and}\;\;\left(\ppxkk,\dots,\ppxdd\right)$$
    along $A$ provide the given trivializations of the bundles $K_-$
    and $K_+$;
  \item $df$ and ${\cal R}(d\wt f)$ are homotopic via a homotopy of
    epimorphisms fixed near $\p U$.
  \end{itemize}
\end{lemma}
\begin{proof}
  There exists splittings $U_1=C\times[-2,2]\times D^d$, where $D^d$
  is the unit $d$-disc, and $U_2=\oC\times[-2,2]$ of neighborhoods
  $U_1\supset C=C\times 0\times 0$ in $U$ and $U_2\supset\oC=f(C)$ in
  $X$, such that the map $f$ has the form
  $$C\times[-2,2]\times [-1,1]^d\ni (v,z,x=(x_1,\dots, x_d))\mapsto
  (\ov=f(v),z)\in \oC\times[-2,2].
  $$ 
  Consider a $ C^\infty$ function $\lambda:[-2,2]\to[-2,2]$ which
  coincides with $z^3-3z$ on $[-1,1]$, with $z$ near $\pm 2$, and such
  that $\pm 1$ are its only critical points. Take two cut-off
  $C^\infty$-functions $\alpha,\beta:\R^+\to[0,1]$ such that
  $\alpha=1$ on $[0,\frac 14]$, $\alpha=0$ on $[\frac12,1]$, $\beta=1$
  on $[0,\frac12]$ and $\beta=0$ on $[\frac34, 1]$.  Set $Q(x)=-
  \sum^k_1 x^2_i + \sum_{k+1}^{d} x_j^2 $ and define first a map $\wh
  f:U_1\to U_2$ by the formula
  $$
  \wh f(v,z,x)=\left(\ov=f(v),\alpha(|x|)\lambda(z)+(1-\alpha(|x|))z
    +\beta(|x|)Q(x)\right),
  $$ 
  and then extend it to $M$, being equal to $f$ outside $U_1\subset
  U$.  The regularized differential of a linear deformation between
  $f$ an $\wh f$ provides the required homotopy between $df$ and
  ${\cal R}(d\wt f)$.
\end{proof}
\begin{remark}\label{rem:killing}{\rm
  It can be difficult to eliminate a double $C$-fold.  Even in the case
  $n=1$   this  is
  one of the central problems of Morse theory.} \end{remark}

\mn A map $f:M \rightarrow X$ is called {\it special folded}, if there
exist disjoint open subsets $U_1, \ldots, U_l \subset M$ such that the
restriction $f|_{M\setminus U}, \,\, U=\bigcup^l_1 U_i, $ is a
submersion (i.e.\ has rank equal $n$) and for each $i=1, \ldots, l$
the restriction $f|_{U_i}$ is a spherical double fold. {\it In
  addition,} we require that the images of all fold components bound
balls in $X$.

\mn The singular locus $\Sigma(f)$ of a special folded map $f$ is a
union of $(n-1)$-dimensional double spheres $S^{n-1}\times
S^0_{\,(i)}=\Sigma(f|_{U_i}) \subset U_i$. It is convenient to fix for
each double sphere $S^{n-1}\times S^0_{\,(i)}$ the corresponding
annulus $S^{n-1}\times D^1_{\,(i)}$ which spans them.  Notice that although the
restriction of $f$ to each annulus $S^{n-1}\times\Int D^1_{\,(i)}$ is an
{\it embedding}, the restriction of $f$ to the union of all the annuli
$S^{n-1}\times \Int D^1_{\,(i)}$ is, in general, only an {\it immersion},
because the images of the annuli may intersect each other.  Using an
appropriate version of the transversality theorem we can arrange   by a $C^{\infty}$-small
perturbation of $f$ that all combinations of images of its fold components intersect transversally.  The differential $df:TM \rightarrow TX$ can be
regularized to obtain an epimorphism ${\cal R} (df):TM \rightarrow
TX$.  To get ${\cal R} (df)$ we regularize $df|_{U_i}$ for each double
fold $f|_{U_i}$.

\mn In our proof of Theorem \ref{thm:main1} we will use the following
result about special folded maps.

\begin{theorem}\label{thm:special-folded}
  {\bf [Special folded mappings]}\label{thm:special} Let $F: TM
  \rightarrow TX$ be an epimorphism which covers a map $f:
  M\rightarrow X$.  Suppose that $f$ is a submersion on a neighborhood
  of a closed subset $K \subset M$, and that $F$ coincides with $df$
  over that neighborhood.  Then if $d>0$ then there exists a special folded map $g:
  M \rightarrow X$ which coincides with $f$ near $K$ and such that
  ${\cal R}(dg)$ and $F$ are homotopic rel. $TM|_K$.  Moreover, the
  map $g$ can be chosen arbitrarily $C^0$-close to $f$ and with
  arbitrarily small double folds.
\end{theorem}

Let us stress again the point that a
   special folded map $f: M \to X$ by definition has only spherical
  double folds,   each fold component $C \subset M$ is
  a sphere whose image  $\oC \subset X$  is embedded  and bounds a ball in
  $X$. 
 \begin{remark}{\rm In the equidimensional case ($d=0$) it is not  possible, in general, to make images of fold components embedded. See Appendix A below.
 }
 \end{remark}
 
\mn Theorem \ref{thm:special} is a modification of the Wrinkling
Theorem from \cite{EM97}.  We formulate the Wrinkling Theorem (see
Theorem \ref{thm:wrinkled}) and explain how to derive
\ref{thm:special} from \ref{thm:wrinkled} in Appendix A below.
 
\medskip

Special folded mappings give a nice representation of (unstable)
formal fibrations.  As a class of maps, it turns out to be too small
for our purposes.  Namely, we wish to perform certain constructions
(e.g.\ surgery) which does not preserve the class of special folded
maps.  Hence we consider a larger class of maps which allow for these
constructions to be performed, and still small enough to admit a
homotopically canonical extension to a formal fibration.  This is the
class of enriched folded maps.

\subsection{Enriched folded maps and their
  suspensions}\label{sec:enriched}

Recall that a folded map is, by definition, a map $f: M^{n+d} \to X^n$
which locally is of the form (\ref{eq:fold}).    In this section we study a certain extra
structure on folded maps which we dub \emph{framed membranes}.

\begin{definition}\label{defn:membrane}
  Let $M^{n+d},X^n$ be closed manifolds and $f: M \to X$ be a folded
  map.  A \emph{framed membrane} of index $k$, $k=0,\dots, d$, for $f$ is   
  a compact, connected $n$-dimensional
  submanifold $V  \subset  M$ with boundary  $\p V=V \cap \Sigma(f)$, together with a {\it framing} $K=(K_-, K_+)$ where  $K_-, K_+$ are {\it trivialized} subbundles of $(\Ker\, df)|_V$ of dimension $k$ and $d-k$, respectively, such that
   
  \begin{enumerate}[(i)]
 
  \item the restriction $f|_{\Int V}: \Int V \to X$ is an embedding;
  \item $K_\pm$ are transversal to each other and to $TV$;
  \item  there exists  a co-orientation of the image $\oC$ of each fold component $C\subset\p V$   such that $K_\pm|_{C} \subset \Cone_\pm(C)$;
  \end{enumerate}
\end{definition}
Thus,  over $\Int V$ we have $ Ker\,df =K_-\oplus K_+$,  while  over $\p V$   
 $ Ker\,df $   splits as $K_-\oplus K_+\oplus\lambda$, where $\lambda=\lambda(C)$ is a line bundle contained in $\Cone_+(C)\cup\Cone_-(C)$.

A  boundary component of a membrane $V$ is called  {\it positive} if  $\lambda(C)\subset\Cone_+ (C)$, and {\it negative} otherwise. We will denote by $\p_+(V,K)$ and $\p_-(V,K)$, respectively,  the union of  positive and negative boundary components of $V$.
  Note that the coorientation of a component $\oC\subset \p\oV$ implied by the definition of a framed membrane is given by inward normals to $\p\oV$ if $C\subset\p_+(V,K)$, and by outward normals to $\p\oV$ if
$C\subset\p_-(V,K)$.
The index of  the fold component $C$ is equal  to $k$ in the former case, and to    $k+1$  in the latter one.

We will call  a framed  membrane $(V,K)$ {\it pure}  if either $\p_+(V,K)=\varnothing$, or $\p_-(V,K)=\varnothing$. Otherwise we call  it  {\it mixed}.

Switching the roles of the subbundles $K_+$ and $K_-$  gives a {\it dual} framing $\oK=(\oK_-=K_+,\oK_+=K_-)$. The index of the framed membrane $(V,\oK)$ equals $d-k$, and we also have
$\p_\pm(V,\oK)=\p_\mp(V.K)$.

 \begin{definition}\label{defn:enrichment}
  An \emph{enriched} folded map is a pair $(f,\fe)$ where $f: M \to X$
  is a folded map and $\fe$ is an \emph{enrichment} of $f$, consisting
  of finitely many disjoint framed membranes $(V_1,K_1), \dots, (V_N,K_N)$
  in $M$ such that $\partial V = \Sigma(f)$, where $V$ is the
  union of the $V_i$.  \end{definition}
We point out that while the definition implies $f$ is injective on
each $V_i$, the images $f(V_i)$ need not be disjoint. 

\begin{example}\label{ex:double-fold}
  {\rm The double $C$-fold $w_C(n+d,n,k)$ defined by
    \eqref{eq:double-fold2} has the annulus $A=C\times D^1\times
    0\subset V\times\R\times\R^d$ as its membrane. Together with
    the frame $K=(K_-,K_+)$, where the subbundles $K_-$ and $K_+$ are generated, respectively by 
    $$\left(\ppxx,\dots,
      \ppxk\right)\;\;\hbox{and}\;\;\left(\frac{\p}{\p x_{k+1}},\dots,
      \frac{\p}{\p x_{d}}\right) $$ along $A$  the membrane defines a
    {\it canonical enrichment} of the double $C$-fold $w_C(n+d,n,k)$.
    In particular, {\it  any special folded map has a canonical enrichment.
    }  Note that we have $\p_+A=A\times(-1)$ and $\p_-A=A\times 1$.
  }
\end{example}

 From the tubular neighborhood theorem and a parametrized version of
Morse's lemma, we get

\begin{lemma}\label{lemma-replacing-eq:14}
 Let $(V,K)$
  be a connected framed membrane for an enriched folded map  $(f,\eps)$. Let $C $ be a connected component  of $V$, and $\oC=f(C)$   its image.   Then
  there is a tubular neighborhood $U_C \subset M$ of $C$ with
  coordinate functions
  \begin{align*}
    (y,u,x): U_C \to C \times \R \times \R^d
  \end{align*}
  and a tubular neighborhood $U_{\overline{C}} \subset X$ with
  coordinate functions 
  \begin{align*}
    (y,t): U_{\overline{C}} \to C \times \R,
  \end{align*}
  such that  we have
  \begin{description}
  \item{-} $\frac{\p}{\p u}\in\lambda(C)$  along $C$;
  \item{-} the vector field $\frac{\p}{\p t}$ defines the coorientation of $\oC$ implied by the framing of $V$;
  \item{-} the vector fields  $\frac{\p}{\p x_1}\big|_{V\cap\Op C},\dots, \frac{\p}{\p x_k}\big|_{V\cap\Op C}$   belong to $K_-$ and provide  its  given trivialization, while 
   $\frac{\p}{\p x_{k+1}}\big|_{V\cap\Op C},\dots, \frac{\p}{\p x_d}\big|_{V\cap\Op C}$    provide  the given trivialization of $K_+$;
\item{-} in these local coordinates, $f$ is given by
  \begin{align*}
    f(y,u,x) = (y,t(x,u)),
  \end{align*}
  where $$t(x,u) = Q_k(x) \pm u^2=   - \sum_{i=1}^k x_i^2 + \sum_{i=k+1}^d x_i^2\pm u^2,$$ 
  and   $V\cap U_C$ coincides   with $\{x=0, \pm u\geq 0\}$, where the signs in the above formulas coincide with the sign of  the boundary component $C$ of the framed membrane $(V,K)$.
  \end{description}
\end{lemma}
\begin{remark} {\rm While the use of    normal form \ref{lemma-replacing-eq:14} is convenient but  it is not necessary.
Indeed, it   is obvious that all the stated properties can be achieved by a $C^1$-small perturbation of our data near $C$, and this will  suffice  for our purposes.}
\end{remark}
 
  \medskip
  
A \emph{suspension} of a folded map $f:M\to X$ is a surjective
homomorphism $\Phi :TM\oplus\eps^1\to f^*TX\oplus\eps^1$ such that
$\pi_X\circ \Phi|_{TM}\circ i_M=df $, where $\pi_X$ is the projection
$TX\oplus\eps^1\to TX$ and $i_M:TX\to TX\oplus0 \hookrightarrow
TX\oplus\eps^1$ is the inclusion.  The main reason for considering
enrichments of folds is that an enriched folded map  admits a suspension
whose  homotopy class depends only on the enrichment.
\begin{proposition}\label{prop:canonical-enrichment}
  To an enriched folded map $(f,\fe)$ we can associate a homotopically
  well defined suspension $\cL(f,\fe)$.
\end{proposition}

\begin{proof}
  The suspension $TM\oplus \epsilon^1 \to f^* TX \oplus \epsilon^1$
  will be of the form
  \begin{align}\label{eq:21}
    \begin{pmatrix}
      df & X\\
      \alpha & q\\
    \end{pmatrix},
  \end{align}
  where $\alpha: TM \to \epsilon^1$ is a 1-form, $X$ is a section of
  $f^* TX$, and $q$ is a function.

  The 1-form $\alpha$ is defined as $\alpha = du$ near $\partial V$,
  using the local coordinate $u$ on $U_C$.  To extend it to a 1-form
  on all of $M$, we  will extend the function $u$.  We first construct
  convenient local coordinates near $V$.

  The map $f: M \to X$ restricts to a local diffeomorphism on
  $\Int(V)$.  The local coordinate functions $ x=(x_1,\dots, x_d)$  near $\p V$ from the normal form \ref{lemma-replacing-eq:14} extend to $\Op V$ in such a way that
      the vector fields
  $\frac{\partial}{\partial x_1}, \dots, \frac{\partial}{\partial
    x_k}$ along $V$ generate the bundle $K_-$, while  $\frac{\partial}{\partial x_{k+1}}, \dots, \frac{\partial}{\partial
    x_k}$ along $V$ generate the bundle $K_+$.   The tubular neighborhood theorem
  then gives a neighborhood $U_{\Int V} \subset M$ with coordinate
  functions
  \begin{align}\label{eq:20}
    (\wt y,x): U_{\Int V} \to \Int V \times \R^d
  \end{align}
  so that the fibers of $\wt y$ are the fibers of $f$, or, more precisely, $f(y,x)=f(y,0)$.
   (The   function $\wt y$ and the     function $y$ of the Lemma \ref{lemma-replacing-eq:14}   are
  not directly related; in fact they have codomains of different
  dimension.)  On this overlap, the function $u$ of
  Lemma~\ref{lemma-replacing-eq:14} can be expressed as a function of
  the local coordinates~\eqref{eq:20}.  Indeed, we have $f(\wt y ,x) = Q(x)
  \pm u(\wt y,x)^2$, so
  \begin{align*}
    \pm u(\wt y,x)^2 = f(\wt y,x) - Q(x) = f(\wt y,0) - Q(x) =  \pm u(\wt y ,0)^2 -
    Q(x)
  \end{align*}
  so
  \begin{align*}
    u(\wt y,x) = \sqrt{u(\wt y,0)^2 \mp Q(x)}.
  \end{align*}
  If we Taylor expand the square root, we get
  \begin{align}\label{eq:15}
    u(\wt y,x) = \gamma(\wt y) \mp \delta(\wt y) Q(x) + o(|x|^2).
  \end{align}
  for positive functions $\gamma,\delta$.  The function $Q$ extends
  over $U_{\Int V}$ (use the same formula in the local coordinates of
  the tubular neighborhood of $\Int V$), and hence we can also extend
  $u$ to a neighborhood of $V$ inside $U_V = U_C \cup U_{\Int V}$,
  such that on $U_{\Int V}$ is satisfies~(\ref{eq:15}).  Extend $u$ to
  all of $M$ in any way, and let $\alpha = du$.

  We have defined a bundle map $(df,\alpha): TM \to f^* TX \oplus
  \epsilon^1$ which is surjective whenever $\alpha |_{\Ker df} \neq
  0$.  Near $V$, $\alpha |_{\Ker df} = 0$ precisely when $x = 0 \in
  \R^d$.  It remains to define the section $(X,q)$ of $f^* TX \oplus
  \epsilon^1$.  Pick a function $\theta: U_V \to [-\pi,\pi]$ such that
  $u = - \sin \theta$ near $\partial V$, is negative on $\Int V$ and
  equal to $-\pi$ on $V - U_C$, and which is equal to $\pi$ outside a
  small neighborhood of $V$.  Then set
  \begin{align}\label{eq:suspension}
    X(u) &= (\cos\theta) \frac{\partial}{\partial u}\\
    q(u) &= \sin\theta
  \end{align}
\end{proof}
\begin{remark}\label{rem:neg-stab} 
{\rm Changing the sign of $X(u)$ in the formula \eqref{eq:suspension}
provides another suspension of the enriched folded map $(f,\fe)$ which we will denote
by $\mathcal{L}_-(f,\e)$.  If $n$ is even then the two suspensions  $\mathcal{L}(f,\e)$ and $\mathcal{L}_-(f,\e)$ are homotopic.  
}
\end{remark}
 
\begin{remark}
{\rm 
  Most (but not all) of the data of an enrichment $\e$ of a folded map $f: M \to X$
  can be reconstructed from the suspension $\Phi = \mathcal{L}(f,\fe)$. 
  If we write $\Phi$ in the matrix form~\eqref{eq:21}, the manifold
  $V$ is the set of points with $q \leq 0$ and $(df,\alpha): TM \to
  f^*TX \oplus \epsilon^1$ not surjective.  The partition of $\partial
  V$ into $\partial_\pm (V,K)$ is determined by the  coorientation of images of 
  the fold components. On the other hand, the splitting $\Ver=K_+\oplus K_-$    cannot  be
  reconstructed from the suspension.}
\end{remark}

\begin{lemma}\label{lm:special-enrichment}
  Let $f:M\to X$ be a special folded map.  Then the suspension
  $\cL(f,\fe)$ (as well as the suspension $\mathcal{L}_-(f,\fe)$) of the canonically enriched folded map $ (f,\fe)$ is
  homotopic to its stabilized regularized differential ${\mathcal R}
  df$.
\end{lemma}
\begin{proof}
  This can be seen in the local models.
\end{proof}
\subsection{Fold surgery}\label{sec:surgery-general}

As in the previous section, we study a folded map $f: M \to X$ with
cooriented  fold images.  Let $C \subset \Sigma(f)$ be a connected
component of index $k$, and let $\overline{C} \subset X$ be its
image.  For $p \in \oC$, the fibers $f^{-1}(p)$ has a singularity.
There are two directions in which we can move $p$ away from $\oC$ to
resolve the singularity and get a manifold.  The manifold we get by
moving $p$ to the positive side (with respect to the coorientation)
differs from the manifold we get by moving $p$ to the negative side by
a surgery of index $k$, i.e.\ it has an embedded $D^k \times \partial
D^{d-k}$ instead of a $\partial D^k \times D^{d-k}$.
If $\oC$ bounds an embedded domain $\oP\subset X$, then one can try to
prevent the surgery from happening, or, which is the same, to perform
an inverse Morse surgery fiberwise in each fiber $f^{-1}(p),
p\in\oP$. In this section we describe this process, which we call
\emph{fold eliminating surgery} in more detail.
 
\subsubsection{Surgery template}\label{sec:surgery-template}
 
We begin with a local model for the surgery.  Let $P$ be an
$n$-dimensional oriented manifold with collared boundary $\partial P$.
The collar consists of an embedding $\partial P \times [-1,0] \to P$,
mapping $(p,0) \mapsto p$.
Extend $P$ by gluing a bicollar $U= \partial P \times [-1,1]$,
  
\begin{align}\label{eq:19}
  \tilde P = P \mathop{\cup}\limits_{\partial P \times [-1,0]} U.
\end{align}

Let $Q $ be a quadratic form of index $k$ on $\R^{d+1}$:
 
\begin{align*}
  Q(x)& =-||x_-||^2+||x_+||^2,
\end{align*}
where $x=(x_1,\dots, x_{d+1})\in\R^{d+1}, \;x_-=(x_1,\dots, x_k),\;
x_+=(x_{k+1},\dots, x_{d+1})$.

Let $H \subset \R^{d+1}$ be the domain
\begin{align*}
  H  &= \left\{ |Q|\leq 1,\,||x_+||\leq 2\right\}, 
\end{align*}
We are going to use the map $Q: H \to [-1,1]$ as a prototype of a
fold.  The boundary of the (possibly singular) fiber $Q^{-1}(t)$ 
can be identified with $S^{k-1} \times S^{d-k-1}$
via the diffeomorphism
\begin{align}\label{eq:22}
  \{Q = t, \|x_+\| = 2\} & \to S^{k-1} \times S^{d-k-1}\\
  (x_-, x_+) & \mapsto \left(\frac{x_-}{\|x_-\|},
  \frac{x_+}{\|x_+\|}\right).
\end{align}

The map $\Id\times Q :\wt P\times H\to \wt P\times[-1,1]$ is a folded
map which has $P\times 0\subset P\times H$ as its fold of index $k $
with respect to the coorientation of the fold defined by the second
coordinate of the product $P\times [-1,1]$.
\begin{figure}[hi]
  \centerline{\includegraphics[height=40mm]{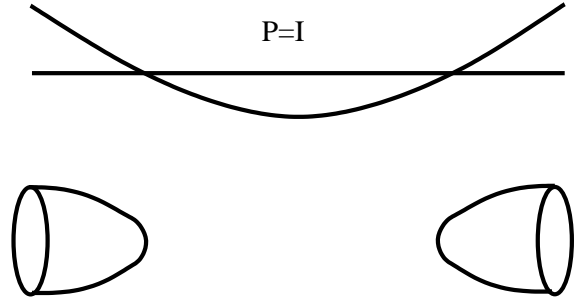}}
  \caption{$\wt P^\varphi=\{Q(x)=x_1^2+x_2^2=\varphi(p)\}$}
  \label{mw16}
\end{figure} 
Given a smooth function $\varphi:\wt P\to[-1,1]$ we define
\begin{equation}
\label{eq:level}
\begin{split}
  \wt P ^\varphi&=\{(p,x)\in \wt P\times H \,|\, Q
  (x)=\varphi(p)\}\,,\cr Z^\varphi &=\{(p,x)\in \p \wt P\times H \,|\,
  Q (x)=\varphi(p)\}\,,\cr R^{\varphi}
  &=P^{\varphi}\cap\{||x_+||=2\}.\cr
 \end{split}
\end{equation}
We then have $\p\wt P^{\varphi} =Z ^{\varphi}\cup R^{\varphi}$.

Given a one-parameter family of functions $\varphi_t:\wt P\to[-1,1]$,
$t\in[0,1]$, we denote
\begin{equation}\label{eq:interpolation}
  \begin{split}
    \wt P^{\varphi_t}&=\{(p,x,t)\in \wt P\times H \times[0,1]\,|\,   Q (x
    )=\varphi_t(p)\}\,,\cr
    Z^{\varphi_t} &=\{(p,x,t)\in \p \wt P\times H \times[0,1]\,|\,  Q
    (x )=\varphi_t(p)\}\,,\cr 
    R^{\varphi_t} &=P^{\varphi_t}\cap \{||x_+||=2\}.\cr
\end{split}
\end{equation}
We have $\p\wt P^{\varphi_t}=Z^{\varphi_t}\cup R^{\varphi_t}\cap \wt
P^{\varphi_0}\cup\wt P^{\varphi_1}$.  We will consider the projection
\begin{align*}
  \pi: \tilde P \times H \times [0,1] \to \tilde P \times [0,1]
\end{align*}
and especially its restriction to the
subsets~\eqref{eq:interpolation}.  Using~\eqref{eq:22}, the set
$R^{\varphi_t}$ can be identified with $(\tilde P \times [0,1]) \times
S^{k-1} \times S^{d-k-1}$ via a diffeomorphism over $\tilde P \times
[0,1]$.  In particular we get a diffeomorphism
\begin{align}\label{eq:23}
  R^{\varphi_0} \times [0,1] \to R^{\varphi_t},
\end{align}
which scales the $x_-$ coordinates.  In fact it can be seen to be
given by the formula
\begin{align*}
  (p,(x_-, x_+), t) \mapsto
  (p,(\sqrt{\frac{4-\varphi_t(p)}{4-\varphi_0(p)}} x_-, x_+),t),
\end{align*}
although we shall not need this explicit formula.

The restriction $\pi: \tilde P ^{\varphi_t} \to \tilde P \times [0,1]$
is our ``template cobordism'', and we record its properties in a lemma.
\begin{lemma}\label{lm:model1}
  \begin{description}
  \item{a)} Suppose that $0$ is not a critical value of
    $\varphi$. Then $\wt P^\varphi$ is a smooth manifold of dimension
    $n+d$, and the projection $\pi|_{\wt P^\varphi}:\wt P^\varphi\to
    \wt P $ is a folded map with the fold $C = \tilde P ^\varphi \cap
    (\tilde P \times \{0\})$, which projects to
    $\oC=\pi(C)=\varphi^{-1}(0) \subset \tilde P$.  In particular,
    the map $\pi|_{\wt P^\varphi}$ is non-singular if $\varphi$ does
    not take the value $0$.  The fold $C$ has index $k$ with respect
    to the coorientation of $\oC$ by an outward normal vector field to
    the domain $\{\phi\leq0\}$.
  \item{b)} Let $\varphi_t:\wt P\to[-1,1]$, $t\in[0,1]$, be a
    one-parameter family of functions such that $0$ is not a critical
    value of $\varphi_0$ and $\varphi_1$ and of the function $ \wt
    P\times[0,1]\to[-1,1] $ defined by $(p,t)\mapsto \varphi_t(p)$,
    $p\in\wt P, t\in[0,1]$.  We also assume $\phi_t(p)$ is independent
    of $t$ for $p$ near $\partial P$.  Then $\pi|_{\tilde
      P^{\varphi_t}}: \tilde P^{\varphi_t} \to \tilde P \times [0,1]$
    is a folded cobordism between the folded maps $\wt
    P^{\phi_0}\to\wt P$ and $\wt P^{\phi_1} \to \wt P$.  We have
    $Z^{\phi_t} = Z^{\phi_0} \times [0,1]$, so together
    with~\eqref{eq:23} we get a diffeomorphism
    \begin{align}\label{eq:24}
      (Z^{\phi_0} \cup R^{\phi_0}) \times [0,1] \to Z^{\phi_t} \cup
      R^{\phi_t}.
    \end{align}
  \end{description}
\end{lemma}

We will need to apply the above lemma to two particular functions on
$\wt P$.  Recall that $U = \partial P \times [-1,1] \subset \tilde
P$ denotes the bicollar.  Take $\varphi_0\equiv 1$ and pick a function
$\varphi_1$ with the following properties:
\begin{itemize}
\item $\phi_1 = 1$ near $\partial \tilde P$,
\item $\phi_1 = -1$ on $\tilde P \setminus U$,
\item For $(p,v) \in U$, $\phi_1(p,v)$ is a non-decreasing function of
  $v$,
\item $\phi_1(p,v) = v$ for $|v| <.5$.
\end{itemize}
 
We will write $\wt P^0$ and $\wt P^1$ for $\wt P^{\varphi_0}$ and $\wt
P^{\varphi_1}$, and denote by $\pi^0$ and $\pi^1$ the respective
projections $\wt P^0\to\wt P$ and $\wt P^1\to\wt P$. Similarly, we
will use the notation $Z^0, Z^1, R^0$ and $R^1$ instead of
$Z^{\varphi_0}, Z^{\varphi_1},R^{\varphi_0}$ and $R^{\varphi_1}$.  The
map $\wt P^0\to\wt P$ is non-singular, while the map $\wt P^1\to\wt P$
has a fold singularity with image $\p P\subset\wt P$. The index of
this fold with respect of the outward coorientation to the boundary of
$P$ is equal to $k$.

Taking linear interpolations between $\varphi_0$ and $\varphi_1$ in
one order or the other, we get folded cobordisms in two directions
between the map $\wt P^0\to\wt P$ and $\wt P^1\to \wt P$. We will
denote the corresponding cobordisms by $\wt P^{01}$ and $\wt P^{10}$,
respectively.  The projections $\pi^{01}:\wt P^{01}\to\wt
P\times[0,1]$ and $\pi^{10}:\wt P^{10}\to\wt P\times[0,1]$ are folded
bordisms in two directions between $\pi^0:\wt P^0\to\wt P$ and
$\pi^1:\wt P^1\to \wt P $.  We think of $\tilde P^{\phi_t}$ as a
one-parameter family of (possibly singular) manifolds, interpolating
between $\tilde P^0$ and $\tilde P^1$.  Using the
trivialization~\eqref{eq:24}, these manifolds all have the same
boundary, so $\pi^{01}$ and $\pi^{10}$ may be used as local models for
cobordisms.  They allow us to create, or annihilate a fold component,
respectively.  We describe the fold eliminating surgery more formally
in the next section and leave the formal description of the inverse
process of fold creating surgery to the reader. In fact fold creating
surgery will not be needed for the proof of the main theorem.
\medskip

In the context of enriched folded maps there are two versions of fold  eliminating surgery.
One will be referred to as  {\it membrane eliminating}. In this case the membrane will be
eliminated together with the fold.  The second one will be referred to
as   {\it membrane expanding.} In that case the
membrane after the surgery will be spread over $\oP$, the image of $P$
in the target.
  
For the  membrane eliminating case we choose  
 the  submanifold $$V_-=\{(x_2,\dots, x_{d+1})=0, x_1\leq 0\}\cap\wt P^{10}\subset \wt P\times H\times [0,1]$$ as a template membrane for the folded bordism
  $\pi^{10}:\wt P^{10}\to\oP$.  Next we choose the subbundles  $K_-$ and $K_+$ 
   spanned by vector fields
$\frac{\partial}{\partial x_2}, \dots, \frac{\partial}{\partial x_{k}}$ and $\frac{\partial}{\partial x_{k+1}}, \dots, \frac{\partial}{\partial x_{d+1}}$, respectively, as  a template framing. Note that with this choice we have  $\p V_-=\p_-(V_-,K)$ and the index of the membrane $V_-$ is equal to $k-1$.

For the membrane expanding surgery we choose  as a template membrane the submanifold 
$$V_+=\{(x_1,\dots, x_k,x_{k+2},\dots, x_{d+1})=0, x_k\geq 0\}\cap\wt P^{10}\subset \wt P\times H\times [0,1]$$ with boundary $\Sigma(\pi^{10})$ as the membrane for  the folded bordism   $\pi^{10}:\wt P^{10}\to\wt P\times[0,1]$.  
We choose the   subbundles  $K_-$ and $K_+$ 
   spanned by vector fields
$\frac{\partial}{\partial x_1}, \dots, \frac{\partial}{\partial x_{k}}$ and $\frac{\partial}{\partial x_{k+2}}, \dots, \frac{\partial}{\partial x_{d+1}}$, respectively,  as  a template framing.
Note that with this choice we have  $\p V_+=\p_+(V_+,K)$
and the index of the membrane $V_+$ is equal to $k$.

 Note that in the membrane eliminating  case the restriction of the membrane $V_-$ to $\wt P^1$  projects diffeomorphically onto $P\subset \wt P$, while in the membrane expanding case  the restriction of the membrane 
   $V_+$ to $\wt P^1$  projects diffeomorphically onto $\wt P\setminus\Int P\subset \wt P$.

\subsubsection{Surgery}\label{sec:surgery-construction}
{\sl Membrane eliminating surgery}. Let $(f:M\to X,\fe)$ be an enriched
folded map and $(V,K)$ one of its membranes.  Suppose that the framed  membrane $(V,K)$ is pure and assume first that
$\p_+(V,K)=\varnothing$, and that the index of the membrane is equal to $k-1\geq 0$. Note that in this case the boundary fold $\p V$ has index $k$ with respect to the outward coorientation of $\p \oV$.
Consider the model enriched folded map $\pi^1:\wt
P^1\to \wt P$ where $P$ is diffeomorphic to $V$. Fix a diffeomorphism
$\psi:P\to \oV=f(V)\subset X$.  Let $U^1$ denote a neighborhood of
$\partial P \subset \tilde P^1$.  According to
Lemma~\ref{lemma-replacing-eq:14} there exist an extension
$\wt\psi:\wt P\to X$ of the embedding $\psi$ and an embedding
$\Psi:U^1\to M$ such that
\begin{itemize}
\item  the diagram
  \begin{equation}
    \xymatrix{U^1\ar[r]^{\Psi}\ar[d]^{\pi^1}&M\ar[d]^f\\
      U \ar[r]^{\wt\psi} &X}
  \end{equation}
  commutes;
\item $\Psi^{-1}(V) = V_- \cap U^1$;
\item the canonical framing of the membrane $V_-\cap U^1$ is sent by
  $\Psi$ to the given framing of the membrane $V$.
\end{itemize}
 
The data needed for eliminating the fold $\partial V$ by surgery
consists of an extension of $\Psi$ to all of $\tilde P^1$ such that
\begin{itemize}
\item  the diagram
  \begin{equation}\label{eq:diagram-neg}
    \xymatrix{\wt P^{1}\ar[r]^{\Psi}\ar[d]^{\pi^1}&M\ar[d]^f\\
      \wt P \ar[r]^{\wt\psi} &X}
  \end{equation}
  commutes;
\item $\Psi^{-1}(V) = V^1_- := V_- \cap \tilde P^1$;
\item the canonical framing of the membrane $V_-^1$ is sent by $\Psi$
  to the given framing of the membrane $V$.
\end{itemize}
 
\begin{construction}\label{constr:membrane-elim}
  Fold eliminating surgery consists of replacing
  $\wt P^1$ by $\wt P_0$ with the projection $\pi^0$.  More precisely,
  cut out $\Psi(\tilde P^1) \times [0,1]$ from $M \times [0,1]$ and
  glue in $\tilde P^{10}$ along the identification~\eqref{eq:24}.
  This gives an enriched folded map $W \to X \times [0,1]$ which is a
  cobordism starting at $f$, and ending in an enriched folded map
  where the fold $\partial V$, {\it together with its membrane } $V$, has been removed.
\end{construction}

\begin{figure}[hi]
  \centerline{\includegraphics[height=40mm]{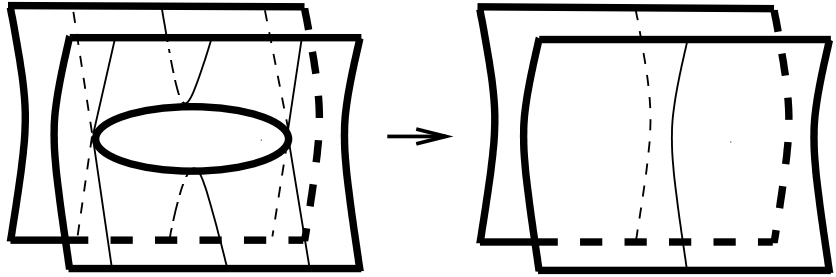}}
  \caption{Fold eliminating surgery $(n=2,d=0)$}
  \label{mw7}
\end{figure}

\begin{figure}[hi]
  \centerline{\includegraphics[height=70mm]{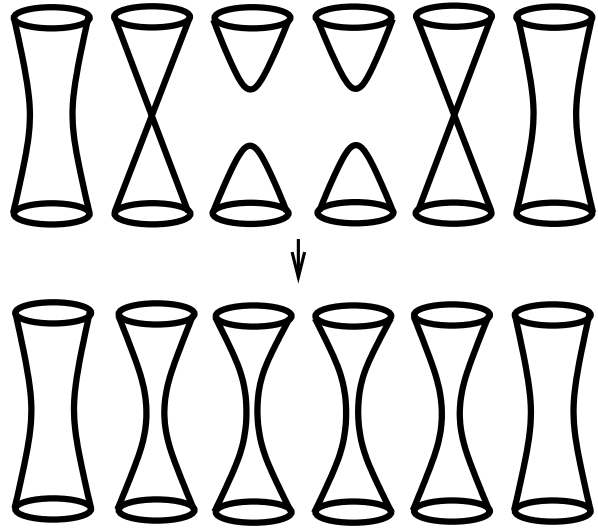}}
  \caption{Fold eliminating surgery $(n=1,d=2)$}
  \label{mw8}
\end{figure}

 The case   $\p_-(V,K)=\varnothing$ can be reduced to the previous one by the following procedure.
 Let $\oK$ be the dual framing of $V$, see Section \ref{sec:enriched} above. Then
 $\p V=\p_-(V,\oK)$, $\p_+(V,\oK)=\varnothing$, $d>k$,  and the membrane $(V,\oK)$ has index $d-k-1$.
 Hence, we can use for the membrane eliminating surgery the  above template of index $d-k$, and then switch the framing of the constructed membrane in the cobordism to the dual one.
 
\medskip {\it Membrane expanding case.}  Let $ P$ be a domain in $X$
with   boundary bounded by  image $\oC$  of   a fold $C$ of index $k$ with respect to the outward orientation
of $\oC=\p P$.  Suppose that the membrane  $V$  adjacent to $C$ projects to the complement of $P$ in $X$,  
 i.e. $\oV=f(V)\subset X\setminus\Int P$, and $C\subset\p_+(V,K)$.
  The case $C\subset \p_-(V,K)$ can be reduced to the positive by passing to the dual framing as it was explained above  in the membrane eliminating case.
 
\medskip
Consider the template enriched folded map $\pi^1:\wt P^1\to \wt P$ as
in Section~\ref{sec:surgery-template}.  Let $\psi$ denote the
inclusion $ P\hookrightarrow X$.  According to
Lemma~\ref{lemma-replacing-eq:14}, there exist an extension
$\wt\psi:\wt P\to X$ of the embedding $\psi$ and an embedding
$\Psi:U^1\to M$ such that
\begin{itemize}
\item  the diagram
  \begin{equation}
    \xymatrix{U^1\ar[r]^{\Psi}\ar[d]^{\pi^1}&M\ar[d]^f\\
      U \ar[r]^{\wt\psi} &X}
  \end{equation}
  commutes;
\item $\Psi^{-1}(V) = V_+\cap U^1$;
\item the canonical framing of the membrane $V_+\cap U^1$ is sent by
  $\Psi$ to the given framing of the membrane $V$.
\end{itemize}

In this case the data needed for eliminating the fold $\partial V$ by
surgery consists of an extension of $\Psi$ to all of $\tilde P^1$ such
that
\begin{itemize}
\item  the diagram
  \begin{equation}\label{eq:diagram-pos}
    \xymatrix{\wt P^{1}\ar[r]^{\Psi}\ar[d]^{\pi^1}&M\ar[d]^f\\
      \wt P\ar[r]^{\wt\psi} &X}
  \end{equation}
  commutes;
\item the canonical framing of the membrane $V_+\cap\wt P^1$ is sent
  by $\Psi$ to the given framing of the membrane $V$.
\end{itemize}

\begin{construction}\label{constr:membrane-exp}
  Fold eliminating surgery consists of replacing $\wt P^1$ by $\wt
  P^0$ with the projection $\pi_0$.  Exactly as in
  Construction~\ref{constr:membrane-elim} we get an enriched folded
  map $W \to X \times [0,1]$ which is a cobordism starting at $f$, and
  ending in an enriched folded map where the fold $\partial V$  has
  been removed.
\end{construction}
 
Both constructions eliminate the fold $\p V$.  The difference between
them is that in the membrane expanding case, the above surgery spreads the
membrane $V$ over the domain $P$.  

\subsubsection{Bases for fold surgeries}\label{sec:basis-surgery}

The embedding $\Psi:\wt P^1 \to M$ required for the surgeries in
constructions~\ref{constr:membrane-elim} and~\ref{constr:membrane-exp}
is determined up to isotopy by slightly simpler data.  To any smooth
manifold $P$ with collared boundary, let $\varphi_1: P \to [-1,1]$ be
the function defined in Section~\ref{sec:surgery-template} and
let
\begin{align*}
  S^{k-1}P = \{(p,x_-) \in \tilde P \times D^k |\; \varphi_1(p) = -
  \|x_-\|^2\}.
\end{align*}
This is a closed manifold, which up to diffeomorphism depends only on
$P$.  In fact, it is diffeomorphic to the boundary of $P \times D^k$
(after smoothing the corners of $P \times D^k$).  The projection
$(p,x_-) \mapsto p$ restricts to a folded map $\pi:S^{k-1}P \to P$
with fold $\partial P$ of index $k$ with respect to the outward
orientation of the boundary of $P$.  We have an embedding
\begin{align*}
  S^{k-1} P \to \wt P^1 \subset \tilde P \times H
\end{align*}
given by $(p,x_-) \mapsto (p,x_-,0)$.  The normal bundle of this
embedding has a canonical frame given by projections of the frame 
 $$\frac{\p}{\p x_+}= \left(
 \frac{\partial}{\partial x_{k+1}},
\dots, \frac{\partial}{\partial x_{d+1}} \right)$$ to $TS^{k-1}P$.

We also have an embedding
$\partial P \to S^{k-1}P$ as $p \mapsto (p,0)$, which identifies
$\partial P$ with the folds of the projection $S^{k-1}P \to P$, and
the normal bundle of $\partial P \subset S^{k-1}P$ is framed by

 $$\frac{\p}{\p x_-}= \left(
  \frac{\partial}{\partial x_{1}},
\dots, \frac{\partial}{\partial x_{k}}\right). $$
 
Let us denote $P_-=V_-\cap S^{k-1}P$. Thus we
have $$P_-=S^{k-1}P\cap\{(x_2,\dots, x_k)=0, x_1\leq 0\}=\{(x_2,\dots,
x_k)=0, x_1=-\sqrt{-\varphi_1(p)}\}.$$
 
\begin{definition}
  Let $(f: M^{n+d} \to X^n, \fe)$ be an enriched folded map.
 
  a) {\bf Membrane eliminating case.}  Let $(V,K)\subset M$ be a framed  membrane
  with $\p_+(V,K)=\varnothing$. A \emph{basis} for membrane--eliminating
  surgery consists of a pair $(h: S^{k-1}P \to M,\mu)$, where $P$ is a
  compact $n$-manifold with boundary, $h: S^{k-1}P \to M$ is an
  embedding, and $\mu = (\mu_{k+1}, \dots, \mu_{d+1})$ is a framing of
  the normal bundle of $h$, such that the following conditions are
  satisfied.
  \begin{itemize}
  \item $h(P_-)= V$;
  \item the map $f\circ h$ factors through an embedding $g:P\to X$,
    i.e.\ $f\circ h=g\circ\pi$, and hence $f(h(S^{k-1}P))=\overline
    V=f(V)$;
  \item the vectors $\mu_{k+1},\dots, \mu_{d+1}$ belong to $\Ker
    df|_{h(S^{k-1}P)}$ and along $h( P_-)$ coincide with the given
    framing of the bundle $\wt\Ker_+df$.
  \item the vectors $dh(\frac\p{\p x_2}),\dots,dh(\frac\p{\p x_k})$
    defines the prescribed framing of the bundle $\wt\Ker_-|_V$;
  \item $h(S^{k-1}P)$ is disjoint from membranes of $\e$, other than $V$.
  \end{itemize}
  \begin{figure}[hi]
  \centerline{\includegraphics[height=30mm]{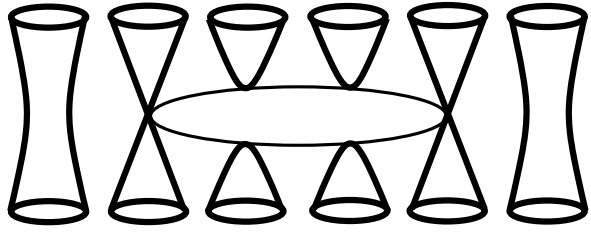}}
  \caption{The oval is the image $h(S^{k-1}P)$, where $k=1$, $P=I$}
  \label{mw14}
\end{figure}
  \smallskip b) {\bf Membrane expanding case.}  Let $ P$ be a domain
  in $X$ bounded by folds of index $k$ with respect to the outward
  orientation of $\oC=\p P$.  Let $C$ be the union of the
  corresponding fold components, and $(V,K)$ the union of framed membranes
  adjacent to $C$.  Suppose that $C\subset\p_+(V,K)$ and $f(V)\subset
  X\setminus\Int P$.  A \emph{basis} for membrane--expanding surgery
  consists of a pair $(h: S^{k-1}P \to M,\mu)$, where $h: S^{k-1}P \to
  M$ is an embedding, and $\mu = (\mu_{k+1}, \dots, \mu_{d+1})$ is a
  framing of the normal bundle of $h$, such that the following
  conditions are satisfied.
  \begin{itemize}
  \item the map $f\circ h$ factors through an embedding
    $g:P\hookrightarrow X$, i.e.\ $f\circ h=g\circ\pi$;
  \item $dh(\frac{\p}{\p x_+}\big|_{\p P})\subset\Ker_-df$ and
    coincides with the given framing of the bundle $\Ker_-$ over the
    fold $C=h(\p P)$;
  \item $h(S^{k-1}P)\cap V=C$ and $h(S^{k-1}P)$ is disjoint from any
    other membranes different from $V$.
  \item the vector field $$dh\left((-1)^k\frac\p {\p
        x_{k+1}}\right)\Big|_{C}$$ is tangent to $V$ and inward
    transversal to $\p V$.
\end{itemize}

\end{definition}
Note that in both cases it follows from the above definitions that
$f\circ h:S^{k-1}P\to X$ is a folded map with the definite fold
$\Sigma(f\circ h)=h(\p P)$.  \medskip

Given a basis $(h,\mu)$, one can extend, uniquely up to isotopy the
embedding $h$ to an embedding $\Psi:\wt P^1\to M$ such that the
diagram \eqref{eq:diagram-neg} or \eqref{eq:diagram-pos}
commutes. This, in turn, enables us to perform a membrane eliminating
or membrane expanding surgery.

\begin{remark}\label{rem: fold-creating}{\it Fold creating surgeries.} {\rm
Fold creating surgeries are inverse to  fold eliminating surgeries. For our purposes we will need only one such surgery which creates a fold of index 1 with respect to the {\it inward} co-orientation of its membrane.
A basis of such a surgery  is given by a pair $(h,\mu)$, where $h$ is an embedding $h:P\times\{ -1,1\}\to M$ over a domain $P\subset X$ disjoint from the folds of the map $f$, and $\mu$ is a framing of the vertical bundle $\Ker\, df|_{h(P\times\{-1, 1\})}$. See Figure \ref{mw17}.}
\end{remark}
 \begin{figure}[hi]
  \centerline{\includegraphics[height=70mm]{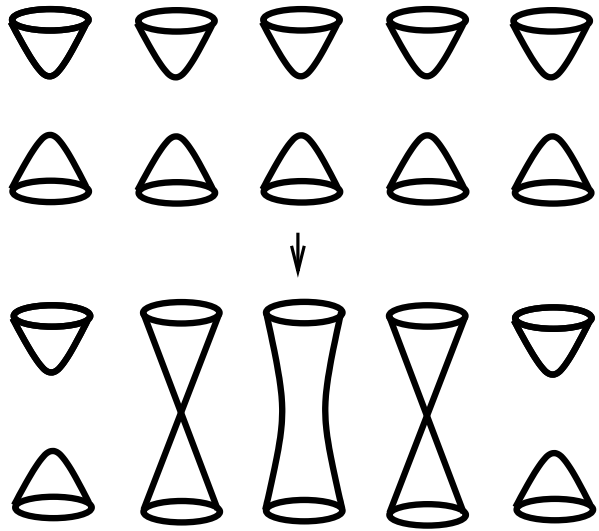}}
  \caption{Fold creating surgery}
  \label{mw17}
\end{figure}
 
\subsubsection{The case $d=2$}\label{sec:surgery-d=2}
Let us review fold eliminating surgeries in the case $d=2$.  These
surgeries can be of index $0$, $1$, $2$ or $3$.  Let $C$ be a union of
fold components whose projections bound a domain $P\subset X$. In the
membrane eliminating case, $P$ is the projection $\oV=f(V)$ of the
membrane which spans $C$. In the membrane expanding case the
membranes adjacent to $C$ projects to the complement of the domain
$P$.  All fold indices below are with respect to the outward
coorientation of the boundary of the domain $ P$.
\begin{description}
\item {\it Index 0.}  We have $S^{-1}P=\p P$, i.e.\ the basis of the
  surgery in this case is a framed embedding $h:\p P\to M$ which sends
  $\p P$ to the fold $C$. Only the  membrane expanding
  surgery is possible in this case.  When a point $p\in X$ approaches
  $\p P$ from outside, a spherical components of the fiber $f^{-1}(p)$
  dies. The surgery prevents it from dying end prolongs its existence
  over all points of $P$.
\item {\it Index 1. }  The surgery basis in this case consists of 2 sections
  $s_\pm: P\to M$, together with framings of the bundle $\Ker\, df$ over
  them. As $p\to \overline{z}\in C$ the sections $s_\pm(p)$ converge
  to the same point $z\in C$, $f(z)=\overline z$. In the membrane
  eliminating case, one of these sections is the membrane $V$.
 
  The manifold $M'$ is obtained by a fiberwise index 1 surgery (i.e.\
  the connected sum) along the framed points $s_+(p)$ and $s_-(p)$, $p\in
  P$. This eliminates the fold $C$ together with the membrane $V$ in
  the membrane eliminating case, and spreads the membrane over $P$ in
  the membrane expanding case. In the latter case  the newly created membrane is a section
  over $P$ which takes values in the circle bundle over $P$ formed
 by central circles of added cylinders $S^1\times[-1,1]$.

\item {\it Index 2.} The surgery basis in this case is a
  circle-subbundle over the domain $ P\subset X$, i.e.\ a family of
  circles in fibers $f^{-1}(p),$ $p\in P$, which collapse to points in
  $C$ when $p$ converges to a boundary point of $P$. In the fold
  eliminating case, the membrane $V$ is a section over $P$ of this
  circle bundle.
 
  The surgery consists of fiberwise index 2 surgery of fibers along
  these circles, which eliminates the fold $C$ together with its
  membrane in the eliminating case, and spreads it over $P$ in the
  expanding one.
\item {\it Index 3.}  The basis of the surgery in this case is a
  connected component of $M$ which forms an $S^2$-bundle over $P$. The
  2-spheres collapse to points of $C$ when approaching the boundary of
  $P$.  The surgery eliminates this whole connected component, in
  particular removing the fold and its membrane. The membrane
  expanding case is not possible for $k=3$.
 
\end{description}
\medskip

\subsection{Destabilization}
\label{sec:destab}

So far (in Theorem~\ref{thm:special-folded} and Lemma
\ref{lm:special-enrichment}) we have related \emph{unstable} formal
fibrations to (enriched) folded maps.  In Theorem~\ref{thm:main1} we
need to work with \emph{stable} formal fibrations (because that is
what $\Omega^\infty\C P^\infty_{-1}$ classifies).  In this section we
study the question of whether an epimorphism $\Phi: TM \oplus
\epsilon^1 \to TX \oplus \epsilon^1$ can be ``destabilized'', i.e.\
homotoped to be of the form $\overline{\Phi} \oplus \mathrm{Id}$ for
some unstable epimorphism $\overline{\Phi}: TM \to TX$.  This is not
possible in general of course (the obstruction is an Euler class).
Instead we prove the following.
\begin{proposition}\label{prop:destabilization}
  Let $\Phi: TM \oplus \epsilon^1 \to TX \oplus \epsilon^1$ be a
  bundle epimorphism with underlying map $f: M \to X$.  Assume $M$ and
  $X$ are connected.  Then there is a compact codimension 0
  submanifold $M_0 \subset M$ which is homotopy equivalent to a
  simplicial complex of dimension at most 1, such that the following
  hold, after changing $f$ and $\Phi$ by a homotopy (in the class of
  bundle epimorphisms).
  \begin{enumerate}[(i)]
  \item $f|_{M_0}$ is folded and has an enrichment $\fe$ such that
    \begin{itemize}
    \item $\Phi|_{M_0} = \cL(f_{|M_0},\fe)$ if $n:=\dim\, X> 1$;  
    \item   $\Phi_{|M_0} = \cL(f|_{M_0},\fe)$ or  $\Phi|_{M_0} = \cL_-(f_{|M_0},\fe)$  in the case $n=1$;
    \end{itemize}
  \item $\Phi$ is integrable near $\partial M_0$, i.e.\ it equals $Df
    \oplus \epsilon^1$ there.
  \item $\Phi$ destabilizes outside $M_0$, i.e.\ it equals
    $\overline{\Phi} \oplus \epsilon^1$ there, for some bundle
    epimorphism $\overline{\Phi}: TM|_{M \setminus M_0} \to TX$.
  \end{enumerate}
\end{proposition}

The strategy of the proof of the proposition is as follows.  First we
forget about (iii) in the proposition, and only worry about how $\phi$
and $\Phi$ looks like on $M_0$.  We prove that this can be done for a
large class of possible $M_0$'s.  After that, we consider the
obstruction to destabilizing $\Phi$ outside $M_0$ without changing it
on $M_0$.  This obstruction is essentially an integer, and we prove
that $M_0$ can be chosen so that the obstruction vanishes.


We first give a local model for the enriched folded map $M_0 \to X$.
Let $I = [-1,1]$ be the interval, and let $K \subset \Int(I^n)$ be a
simplicial complex.  We will consider $I^n \subset I^{n+d}$ as the
subset $I^n \times \{0\}$.  Let $U_0 \subset I^n$ be a regular
neighborhood of $K \subset I^n$, and let $U \subset I^{n+d}$ be a
regular neighborhood of $K \subset I^{n+d}$.  Let $\pi: I^{n+d} \to
I^n$ be the projection.  In order to avoid confusion we will write
$\overline{U}_0 = U_0 \times \{0\} \subset I^{n+d}$, and hence $U_0
= \pi(\overline{U}_0)$.  We can assume that $\pi_{|\partial U}
: \partial U \to I^{n}$ is a folded map, with fold $\partial
\overline{U}_0 \subset \partial U$ which has index 0 with respect to
the inward coorientation of $\partial U_0 \subset U_0$.

  Let $N=\partial U \times [-1,1]$ be  a  bicollar  of $\p U$, i.e.   an embedding $N \to I^{n+d}$,
which maps $(u,0) \mapsto u
\in \partial U$ and  $\partial U \times [-1,0]$ to $U$.  Let $M_0 = U \cup N$.  We construct a folded map
$M_0 \to M_0$ in the following way.  First pick a function $\phi:
[-1,1] \to [-1,1] \times \R$ with the following properties.
\begin{enumerate}[(i)]
\item $\phi(\pm s) = (\pm s,0)$ for $s > .5$;
\item $\phi_1'(s) > 0$ for $s < 0$, $\phi_1'(s) < 0$ for $s > 0$, and
  $\phi_1''(0) < 0$;
\item $\phi_2'(0) < 0$. 
\end{enumerate}
In particular $\phi$ is an immersion of codimension 1, and $\phi_1:
[-1,1]$ is a folded map with fold $\{0\}$.  Extend to an immersion
$\phi: [-1,1] \times \R \to [-1,1]\times \R$ with the property that
\begin{align*}
  \phi(\pm s,t) = (\pm s,\pm t)
\end{align*}
for $s > .5$.  Recall that $N = \partial U \times [-1,1]$ and
construct a codimension 0 immersion $\gamma_0: N \times \R \to N
\times \R$ by
\begin{align*}
  \gamma_0(u,s,t) = (u,\phi(s,t)),
\end{align*}
for $u \in \partial U$.  Extend to a codimension 0 immersion
$\gamma_1: M_0 \times \R \to M_0 \times \R$ by
\begin{align*}
  \gamma_1(x,t) = (x,-t)
\end{align*}
for $x \in M_0 - N$.  For $m \in M_0$, let $\gamma_1(m) \in M_0$ be
the first coordinate of $\gamma_1(m,0) \in M_0 \times \R$.
Differentiating $\gamma_1$ then gives a bundle map
\begin{align}\label{eq:27}
  \Gamma: TM_0 \oplus \epsilon^1 \to TM_0 \oplus \epsilon^1
\end{align}
with underlying map $\gamma: M_0 \to M_0$.  We record some of its
properties in a lemma.
\begin{lemma}\label{lem:local-model}
  \begin{enumerate}[(i)]
  \item The map $\gamma_1$ is homotopic in the class of submersions
    (=immersions) to the map $\Id \times (-1): M_0 \times \R \to M_0
    \times \R$.
  \item Let $M_0 \subset I^{n+d}$ and $\gamma$ and $\Gamma$ be as
    above.  Then $\gamma$ is a folded map.  The fold is $\partial U
    \subset N \subset M_0$, and the image of the fold is also
    $\partial U$.  Near $\partial M_0$, the bundle map $\Gamma$ is
    \emph{integrable}, i.e.\ $\Gamma = D\gamma \oplus \epsilon^1$.
  \item Let $\pi: I^{n+d} \to I^n$ be the projection and define a
    bundle map
    \begin{align*}
      H: TM_0\oplus \epsilon^1 \to T(I^n) \oplus \epsilon^1
    \end{align*}
    by $H = (D\pi \oplus \epsilon^1) \circ \Gamma$.  It covers the
    folded map $h = \pi \circ \gamma: M_0 \to I^n$, which has fold
    $\partial \overline{U}_0 \subset \partial U \subset M_0$.  The
    image of the fold is $\partial U_0 \subset I^n$, and it has
    index 0 with respect to the inward coorientation of $U_0$.
  \item\label{item:1} A membrane for the underlying map $h$ can be
    defined as $V = \overline{U}_0$ with the framing  $K=(K_+=\Span( 
   \partial/\partial x_{n+1}, \dots,
    \partial/\partial x_{n+d}$, and $\Ker_-(V)),K_-=\{0\}$.  This defines an enrichment $\fe$ for $h$.
    Finally, $H$ is integrable over $\partial M_0$ and $H =
    \cL(h,\fe)$.
  \end{enumerate}
\end{lemma}
\begin{proof}
  We leave this as an easy exercise for the reader.  Cf also
  Section~\ref{sec:basis-surgery} and the proof of
  Proposition~\ref{prop:canonical-enrichment}. 
\end{proof}

Let  us also point that we
could equally well have based the construction on a map $\phi_-:
[-1,1] \to [-1,1]\times \R$ defined as $\phi$ above, except that we
replace the condition $\phi_2'(0) < 0$ by $\phi_2'(0) > 0$.  Using
this map as a basis for the construction gives a bundle epimorphism
\begin{align*}
  H_-: TM_0 \oplus \epsilon^1 \to T(I^n) \oplus \epsilon^1
\end{align*}
which also satisfies the conclusion of Lemma~\ref{lem:local-model},
except that (\ref{item:1}) gets replaced by $H = \cL_-(h,\fe)$.

\begin{proof}[Proof of Proposition~\ref{prop:destabilization} (i) and (ii)]
  Let $\Phi: TM \oplus \epsilon^1 \to TX \oplus \epsilon^1$ be as in
  the proposition.  By Phillips' theorem we can assume that $\Phi$ is
  induced by a submersion
  \begin{align*}
    \Psi: M \times \R \to X \times \R.
  \end{align*}
  Pick cubes $D = I^{n+d} \subset M$ and $I^n \subset X$.  We
  regard $M_0 \subset D \subset M$ and let $\pi: D \to I^n$ denote
  the projection to the first $n$ coordinates.  We can assume that
  $\Psi(D \times \R) \subset I^n \times \R$.  The space of
  submersions $D \times \R \to I^n \times \R$ is, by Phillips' theorem,
  homotopy equivalent to $O(n+d+1)/O(d)$ which is connected, and hence
  we may assume that $\Psi_{|D} = \pi \times (-1)$.  (Here we used
  $d\geq 1$.  In the case $d = 0$, we get $O(n+d+1)$ which has two
  path components, but after possibly permuting coordinates on $D =
  I^{n+d}$, we may assume that $\Psi_{|D}$ is in the same path
  component as $\pi \times (-1)$.)  Hence, after a homotopy of $\Psi$
  in the class of submersions, we may assume that $\Psi_{|M_0} = \pi
  \times (-1)$.

  By Lemma~\ref{lem:local-model}(i), we can assume after a further
  deformation of $\Psi$ in a neighborhood of $M_0 \subset D$, that
  it agrees with the map $(\pi\times \R) \circ \gamma_1$.  This proves
  (i) and (ii) in the proposition.
\end{proof}
Remember that the domain $M_0$ is a regular neighborhood of a
simplicial complex $K \subset I^n$.  The local model in
Lemma~\ref{lem:local-model} worked for any such $K$, but in the proof
of (i) and (ii) in the proposition we used that $K$ had dimension at
most 1.  On the other hand, $K$ could still be arbitrary within that
restriction.  It remains to prove that $\Phi$ can be destabilized
outside $M_0$, for a suitable choice of $K$.  There is an obstruction
to doing this, which we now describe.

Let $s$ denote a section of $TM \oplus \epsilon^1$ such that $\Phi
\circ s$ is the constant section $(0,1) \in f^*(TX) \oplus
\epsilon^1$.  This defines $s$ uniquely up to homotopy (in fact $s$ is
unique up to translation by vectors in the kernel of $\Phi$).  Over
$M_0$ the epimorphism $\Phi$ is induced by a composition
\begin{align*}
  M_0 \times \R \xrightarrow{\gamma_1} M_0 \times \R
  \xrightarrow{\mathrm{proj}} X \times \R,
\end{align*}
and on $M_0$ we may choose $s$ so that $D\gamma_1$ takes $s$ to a unit
vector in the $\R$-direction.  Another relevant section is the
constant section $s_\infty(x) = (0,1) \in S(TM \oplus \epsilon^1)$.
We have $s(x) = s_\infty(x)$ for $x \in \partial M_0$.  Our aim is to
change $\Phi$ by a homotopy and achieve $s(x) = s_\infty(x)$ for all
$x$ outside $M_0$.  In each fiber, $s(x) \in S(T_x M \oplus \R) =
S^{n+d}$, so by induction of cells of $M - D$, we can assume that
$s_\infty(x) = s(x)$ outside $D$, since $M-D$ can be built using cells
of dimension at most $n+d-1$.  It remains to consider $D - M_0$.  Let
$s_K$ be the section which agrees with $s$ on $M_0$ and with
$s_\infty$ outside $M_0$.  Thus $s$ and $s_K$ are both sections of
$S(TM \oplus \epsilon^1)$ which equal $s_\infty$ outside $D$.  We
study their homotopy classes in the space of such sections.

Using stereographic projection, the fiber of $S(TM\oplus \epsilon^1)$
at a point $x \in M$ can be identified with the one-point
compactification of $T_x M$.  Hence we can think of sections as
continuous vector fields on $M$, which are allowed to be infinite.
The section at infinity is $s_\infty(x) = (0,1) \in S(T_x M \oplus
\R)$.  In this picture we have the following way of thinking of $s_K$:
For $x \in \partial U$, $s_K(x)$ is a unit vector orthogonal to
$\partial U$ pointing outwards.  Moving $x$ away from $\partial U$ to
the \emph{inside} makes $s_K(x)$ smaller and it gets zero as we get
far away from $\partial U$.  Moving $x$ away from $\partial U$ to the
\emph{outside} makes $s_K(x)$ larger and it gets infinite as we get
far away from $\partial U$.  The section $s_K$ depends up to homotopy
only on the simplicial complex $K \subset I^n$, hence the notation.

\begin{lemma}
  There is a bijection between $\Z$ and sections of $S(TM \oplus
  \epsilon^1)$ which agree with $s_\infty$ outside $D$.  The bijection
  takes $s_K$ to $\chi(K) \in \Z$.
\end{lemma}
\begin{proof}
  The tangent bundle $TM$ is trivial over $D$, so the space of such
  sections is just the space of pointed maps $S^n \to S^n$ and
  homotopy classes of such are classified by their degree, which is an
  integer.
 
  We have assumed $U \subset D = I^{n+d}$ so using the standard
  embedding $D \subset \R^n$ we can work entirely inside $\R^n$.
  The geometric interpretation of $s_K$ given above can then be
  rephrased even more conveniently.  Let $r: U \to K$ be the
  retraction in the tubular neighborhood, and let
   \begin{align}\label{eq:26}
    \tilde s_K(x) = x - r(x)
  \end{align}
  for $x\in U$.  Pick any continuous extension of $\tilde s_K$ to $D$
  with the property that when $x \not\in U$, we have
  \begin{align*}
    \tilde s_K(x) \in (T_xM - \{0\})\cup \{\infty\}.
  \end{align*}
  Up to homotopy there is a unique such extension because we are
  picking a point $\tilde s_K(x)$ in a contractible space.  Then
  $\tilde s_K \simeq s_K$.

  To calculate the degree of the corresponding map we perturb even
  further.  Remember that any simplicial complex $K$ has a standard
  vector field with the following property: The stationary points are
  the barycenters of simplices and the flowline starting at a point
  $x$ converges to the barycenter of the open cell containing $x$.
  Let $\psi_\epsilon: K \to K$ be the time $\epsilon$ flow of this
  vector field, and define a vector field $\hat s_K$ just like $\tilde
  s_K$, except that we replace the right hand side of~\eqref{eq:26} by
  $x - \psi_\epsilon \circ r(x)$ for some small $\epsilon > 0$.

  The resulting vector field vanishes precisely at the barycenters of
  $K$, and the index of the vector field at the barycenter of a
  simplex $\sigma$ is $(-1)^{\dim(\sigma)}$.  The claim now follows
  from the Poincar\'e-Hopf theorem.
\end{proof}

\begin{proof}[Proof of Proposition~\ref{prop:destabilization} (iii)]
  Let us first consider the case $n > 1$.  We have proved that all
  possible sections of $S(TM \oplus \epsilon^1)$ which agree with
  $s_\infty$ outside $D$, are homotopic to $s_K$ for some $K$.  Then we can choose $K$ such that $s_K
  \simeq s$.  Since $s_K$ and $s$ agree on $M_0$ there is a homotopy
  of $s$, fixed over $M_0$, so that $s(x) = s_\infty(x)$ for all $x
  \in M - M_0$.  This homotopy lifts to a homotopy of bundle
  epimorphisms $\Phi: TM \oplus \epsilon^1 \to TX \oplus \epsilon^1$,
  and then (iii) is satisfied.

  For  $n=1$ we may not be able
  to choose a $K \subset X$ with the required Euler characteristic,
  since subcomplexes of 1-manifolds always have non-negative Euler
  characteristic.  However, vector fields of negative index can be
  achieved as $-s_K$, and that is the vector field that arises if we
  use the negative suspension $\cL_-(f,\fe)$.
\end{proof}

\subsection{From formal epimorphisms to enriched folded
  maps}\label{sec:epi-folds}

The following theorem summarizes the results of Section
\ref{sec:prelim}.
\begin{theorem}\label{thm:epi-folds}
  Let $\Phi: TM \oplus \epsilon^1 \to TX \oplus \epsilon^1$ be a
  bundle epimorphism.  Suppose that  $d>0$ and $\Phi$ is integrable in a
  neighborhood of a closed set $A \subset M$.  Then there is a
  homotopy of epimorphisms $\Phi_t: TM \oplus \epsilon^1 \to TX \oplus
  \epsilon^1$, $t \in [0,1]$, fixed near $A$, which covers a homotopy
  $\phi_t: M \to X$, such that $\phi_1: M \to X$ is folded, and
  $\Phi_1 = \cL(\phi_1, \fe)$ for some enrichment $\fe$ of
  $\phi_1$. If $n>1$ then the image $\overline{C} \subset X$  of each    fold component  $C \subset M$  of $\phi_1$ bounds a domain in $X$.
 \end{theorem}
\begin{proof}
  First use Proposition~\ref{prop:destabilization} to make $\phi$
  enriched folded over a domain $M_0$, such that $\Phi$ destabilizes
  outside $M_0$.  Then use Lemma~\ref{lm:special-enrichment} to make
  $(\phi,\Phi)$ special enriched outside $M_0$.
 \end{proof}

\section{Cobordisms of folded maps}\label{sec:cob}

Let us rephrase the results of the previous section more
systematically, and put them in the context of the overall goal of the
paper.  So far we have mainly studied the relation between formal
fibrations and enriched folded maps.  Let us formalize the result.  We
consider various bordism categories  of maps $f: M \to X$ such that $M$ and $X$ are both oriented and $X$ is a
compact manifold, possibly with boundary, and which satisfy the
following two conditions:

\begin{description}
\item{C1.}  $f$ has $\Tinf$ ends, i.e.\ there is (as part of the
  structure) a germ at infinity of a diffeomorphism $j:T_\infty\times
  X\leadsto M$ such that $f\circ j=\pi$, where $\pi$ is a germ at
  infinity of the projection $X\times\Tinf\to X$.  This trivialized
  end will be called the \emph{standard end} of $M$.
\item{C2.} There is a neighborhood $U$ of $\partial X$ such that
  $f^{-1}(U) \to U$ is a fibration (i.e.\ smooth fiber bundle) with
  fiber $\Tinf$.
\end{description}

When $\dim(X)>1$ we will always assume that all fold components are homologically trivial, and in particular,  the image $\oC\subset X$ of any   fold component  $C
\subset M$  bounds a
domain in $X$.  Note that this condition is preserved by all fold surgeries which we discussed above in Section \ref{sec:surgery-general}.  We have been studying the following bordism categories of maps.
\begin{definition}
  \begin{enumerate}[(i)]
  \item $\Fib$ is the category of \emph{fibrations} (smooth fiber
    bundles) with fiber $\Tinf$, which satisfy C1 and C2.
  \item $\Fold^{\$}$ is the category of enriched folded maps, satisfying
    C1 and C2.
  \item $\FFib$ is the category of \emph{formal fibrations}, i.e.\ bundle
    epimorphisms $\Phi: TM \oplus \epsilon^1 \to TX \oplus \epsilon^1$
    with underlying map $f: M \to X$, such that $f$ satisfies C1 and
    C2, and such that $\Phi = df \oplus \epsilon^1$ near
    $f^{-1}(\partial X)$.
 
  \end{enumerate}
  are the objects in  bordism
  categories,   Bordisms  in the categories $\Fib$, $\Fold^{\$}$, and $\FFib$ are required to be trivial (as bordisms, and not as fibrations!) over a
  neighborhood of $\partial X$. 
  \end{definition}

We have functors
\begin{align*}
  \Fib \to \Fold^{\$} \xrightarrow{\cL} \FFib.
\end{align*}
The functor $\Fib \to \Fold^{\$}$ is the obvious inclusion.
Everything we said in Chapter~\ref{sec:prelim} works just as well with
the conditions C1 and C2 imposed, and hence
Proposition~\ref{prop:canonical-enrichment} gives the functor $\cL:
\Fold^{\$} \to \FFib$.

In this setup, the main goal of the paper is to prove that any object
in $\FFib$ is cobordant to one in $\Fib$.  Theorem~\ref{thm:epi-folds}, which also works for manifolds with boundary,
says that $\cL$ is essentially surjective: any object in $\FFib$ is
cobordant to an element in the image of $\cL$.  It remains to see that
any object of $\Fold^{\$}$ is cobordant to one in $\Fib$.  In fact
Theorem~\ref{thm:epi-folds} is a little stronger: any object of
$\FFib$ is cobordant to one in the image of $\cL$ using only
homotopies of the underlying maps, i.e.\ no cobordisms of $M$ and $X$.
In contrast, comparing $\Fib$ to $\Fold^{\$}$ involves changing $M$
and $X$ by \emph{surgery}.  The surgery uses the membranes in the
enrichment, and also makes crucial use of a form of Harer's stability
theorem.

In fact, it is convenient to work with a slight modification of the
category $\Fold^{\$}$.
\begin{definition}
  Let $\wt\Fold^{\$}$ be the category with the same objects as
  $\Fold^{\$}$, but where we  formally add  morphisms which create double folds
  along a submanifold $C$,
  i.e. singularities of the
  form~(\ref{eq:double-fold2}). We also formally add inverses of these morphisms. When $n>1$ we additionally require the
  manifold $C$ to be homologically trivial.  By Example~\ref{ex:double-fold}, this has a
  canonical enrichment, and we formally add morphism in
  $\wt\Fold^{\$}$ in both directions between the map $f: M \to X$ and
  the same map $f':M\to X$ with a double $C$-fold singularity created.
\end{definition}
According to \ref{lm:special-enrichment},
 adding a double $C$-fold together with its canonical enrichment changes $\cL(f,\fe)$
only by a homotopy.  Hence the functor $\cL: \Fold^{\$} \to \FFib$
extends to a functor $\wt\Fold^{\$} \to \FFib$. In the proof of our
main theorem, $\Fold^{\$}$ is just a middle step, and it turns out to
be more convenient to work with the modified category
$\wt\Fold^{\$}$.

We leave it as an exercise to the reader to show that when $n>1$ there is, in fact,
no real difference: if two objects in $\wt\Fold^{\$}$ are cobordant,
then they are already cobordant in $\Fold^{\$}$.  We shall not need
this fact.

We prove that any object of $\wt\Fold^{\$}$ is cobordant to one in
$\Fib$ in two steps.
\begin{definition}
  Let $\Fold_h^{\$}\subset \Fold^{\$}$ be the subcategory where
  objects and morphisms are required to satisfy the following
  conditions.
  \begin{enumerate}[(i)]
  \item Folds are hyperbolic, i.e.\ have no folds of index 0 and 3.
  \item For all $x \in X$, the manifold $f^{-1}(x) - \Sigma(f)$, i.e.\
    the fiber minus its singularities, is connected.
  \end{enumerate}
  Let $\wt\Fold^{\$}_h$ be the category with the same objects, but
  where we formally add morphisms (in both directions) which create double folds.
\end{definition}

Our main result,   Theorem \ref{thm:main1}, follows from Theorem \ref{thm:epi-folds} and  the following two propositions.
\begin{proposition}\label{prop:enriched-hyperbolic}
 Let $d>0$. Any enriched folded map $( f:M\to X,\fe)\in\Fold^{\$}$ is bordant in
  the category $\wt{\Fold}^{\$}$ to an element in $\Fold^{\$}_h$
\end{proposition}
\begin{proposition}\label{prop:hyperbolic-fibrations}
  Let $d =2$.  Any enriched folded map $( f:M\to
  X,\fe)\in\Fold^{\$}_h$ is bordant in the category
  $\wt{\Fold}^{\$}_h$ to a fibration from $ \Fib \subset\Fold^{\$}$.
\end{proposition}
The proof of the latter uses Harer stability.  This is the only part
of the whole story in which $d=2$ is used in an essential way.
The reason for the condition $d>0$ in the first proposition is explained in Appendix A.


\section{Generalized Harer stability theorem}
\label{sec:Harer-generalized}
The main ingredient in the proof of \ref{prop:hyperbolic-fibrations}
is Harer's stability Theorem \ref{thm:Harer-geometric}. In this
section we will  deduce a version of \ref{thm:Harer-geometric} 
for enriched folded maps.

\subsection{Harer stability for enriched folded
  maps}\label{Harer-folds}

The proof of the following main theorem of this section will be given in Section
\ref{sec:proof-Harer}. Section \ref{sec:nodal}   contains necessary preliminary constructions.

\begin{theorem}\label{thm:Harer2}
  Let $(f:M\to X,\fe)\in\Fold_h^{\$}$ be an enriched folded map.  Let $U\subset X$ be a closed
  domain with smooth boundary transversal to the images of the folds
  and $\Sigma_1 \subset \Sigma_2$ be compact surfaces with boundary.
  Let
  \begin{align*}
    j: (\partial U \times \Sigma_2) \cup (U \times \Sigma_1) \to M
  \end{align*}
  be a fiberwise embedding over $\p U$ whose image does not intersect
  any fold or membrane and such that the complement of its image in
  each fiber is connected, even after removing the folds of $f$.

  Then, after possibly changing $(f,\fe)$ by a bordism which is
  constant outside $\Int U$, the embedding $j$ extends to an embedding
  of $U \times \Sigma_2$.

 \end{theorem}

\medskip
   
The following corollary of Theorem \ref{thm:Harer2} will be the key
ingredient in the proof of
Proposition~\ref{prop:hyperbolic-fibrations}.
\begin{corollary}\label{cor:basis-existence}
  Let $(f: M \to X,\fe)\in \Fold^{\$}_h$ be an enriched folded map.
  Let $C_1, \dots, C_K \subset M$ be fold components of $f$ and
  $\oC_1 \dots, \oC_K \subset X$ their image.  Assume that the
  $\oC_i$, $i=1,\dots, K$, are disjoint and their union bounds together a domain $P \subset X$,     all
  folds $C_i$ have the same index with respect to the outward co-orientation
  of $\partial P$ and that one of the following conditions holds:
   \begin{description}
  \item{M1.} there exists a pure framed membrane $V$ which spans
    $C=\bigcup\limits_1^K C_j$ and projects to $P$;
  \item{M2.} all framed membranes adjacent to $C $ project to the complement
    of $P$ and the folds $C_j$ are either all positive or all negative
    (recall that the boundary of a membrane is split into a positive
    and negative part).
  \end{description}
 
  Then $(f,\fe)$ is bordant in the category $\wt{\Fold}^{\$}_h$ to an
  element $(\wt f,\wt\fe)$ such that
  \begin{itemize}
  \item the bordism is constant over the complement of $\Int P$;
  \item  $(\wt f,\wt\fe)$ has no more membranes than $(f,\fe)$;
  \item $\wt f$ admits a basis for a surgery eliminating the fold $C$.
  \end{itemize}
The surgery eliminates the membrane of $C$ in Case M1 and spreads it over $P$ in Case M2.
\end{corollary}
\begin{proof}
  Consider a slightly smaller domain $U\subset \Int P$, so that
  $P\setminus \Int U$ is an interior collar of $\p P$ in $P$.  If the
  index of $C$ with respect to the outward coorientation of $\oC$ is $1$ then the $0$-dimensional vanishing cycles over
  points of $\p U$ form two sections $s_\pm:\p U\to M$ of the map $f$.
  In case M1 we can assume that one of these sections, say $s_-$,
  consists of the points of the membrane $V$. The local structure near
  the membrane allows us to construct fiberwise embeddings
  $S_-:U\times D^2\to M$ and $S_+:\p U\times D^2\to M$ such that
  $S_-|_{U\times 0}$ extends the section $s_-$, $S_+|_{\p U\times
    0}=s_+$ and $S_-(U\times 0)\subset V$.  Applying Theorem
  \ref{thm:Harer2} with $\Sigma_1 = D^2$, $\Sigma_2 = \Sigma_1 \amalg D^2$,
  and $j=S_+ \amalg S_-$, we construct a basis  for a membrane
  eliminating surgery  which removes the fold $C$.  In the case M2 the
  enrichment structure for the membranes adjacent to $C$ provides us
  with an extension of sections $s_\pm$ and $s_+$ to disjoint
  fiberwise embeddings $S_\pm:\p U\times D^2\to M$ such that
  $S_\pm|_{\p U'\times 0}=s_\pm$. To conclude the proof in this case
  we apply \ref{thm:Harer2} with $\Sigma_1 = \varnothing$, $\Sigma_2 =
  D^2 \amalg D^2$, and $j = S_+ \amalg S_-$.  Suppose now that the
  index of $C$ is $2$. Consider first the case M2. Then the vanishing
  cycles over $\p U$ define a fiberwise embedding $\p U\times S^1\to
  M$ over $\p U$ which extends to a fiberwise embedding $j:\p U\times
  A\to M$ disjoint from all folds and their membranes, where $A$ is
  the annulus $S^1\times [-1,1]$.  It follows from the definition of
  the category $ {\Fold}^{\$}_h$ that the complement of the image of
  $j$ is fiberwise connected even after all singularities being
  removed.  Hence we can apply Theorem \ref{thm:Harer2} with $\Sigma_1
  = \varnothing$ and $\Sigma_2 = A$ to construct a basis for a membrane
  expanding surgery  eliminating the fold $C$.  Finally, in the case $M1$ each
  vanishing cycle $j(x\times(S^1\times 0)), x\in\p U$, has a unique
  point $p_x$ which is also in the membrane $V$ of the fold $C$.  This
  point is the center of an embedded disk $D^2 \to A$, and the framed
  membrane gives a fiberwise embedding $j_1: U \times D^2 \to M$ over
  $U$, which over the boundary extends to a fiberwise embedding $j_2:
  \partial U \times A \to M$ over $\partial U$.  Hence, we are  in a position to  apply
  Theorem~\ref{thm:Harer2} with $\Sigma_1 = D^2$ and $\Sigma_2 = A$.
\end{proof}

 \subsection{Nodal surfaces and their unfolding}\label{sec:nodal}
 Consider a  quadratic form
\begin{equation}\label{eq:node}
Q(x)= x_1^2+x_2^2-x_3^2,\;\;  
\end{equation}
Take the handle $H=\{|Q|\leq 1; |x_3|\leq 2\} $ and denote by $K_t$ the
level set $\{Q=t\}\cap H$, $t\in[-1,1]$.  When passing through the
critical value $0$, the level set $K_t$ experiences a surgery of index
$1$, i.e. changes from a $2$-sheeted to a $1$-sheeted hyperboloid. The
critical level set $K_0$ is the cone $\{x_1^2+x_2^2-x_3^2=0, |x_3|\leq 2\}$.
Let us fix diffeomorphisms $\beta_\pm:\p_\pm H=H\cap\{x_3=\pm 2\}\to
S^1\times[-1,1]$ which send boundary circles of $K_t$ to $S^1\times
t$, $t\in I=[-1,1]$. We will write
$\beta_{\pm}(x)=(\beta^S_\pm(x),\beta^I_\pm(x))\in S^1\times[-1,1]$
for $x\in\p_\pm H$.

We will call the singularity of $K_0$ a {\it node} and call surfaces
with such singularities {\it nodal}.  A singular surface $S$ is called
{\it $k$-nodal} if it is a smooth surface in the complement of $k$
points $p_1,\dots,p_k\in S$, while each of these points has a
neighborhood diffeomorphic to $K_0$.
 
\begin{figure}[hi]
\centerline{\includegraphics[height=50mm]{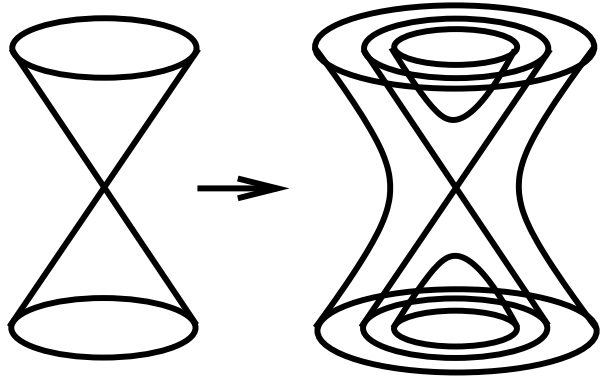}}
\caption{Nodal surface and its unfolding}
\label{mw10}
\end{figure}

The function $Q$ is a folded map $H\to\R$ with the origin $0\in H$ as its fold $\Sigma$.
 If we want to add to it the structure of an {\it enriched} folded map, there should be considered   four local possibilities for the choice of an enriched  membrane   depending on the membrane index and the sign of $\Sigma $
 as the membrane boundary component:
 \begin{enumerate}[(i)]
\item {\sl Index $0$ membrane with negative boundary},
$V_0^-=\{x_1,x_2=0, x_3\leq 0\}$, $K_-=\{0\}, K_+=\Span(\frac\p{\p x_1},\frac{\p}{\p x_2})$;
\item {\sl Index $1$ membrane with positive boundary},
$V_1^+=\{x_2,x_3=0, x_1\geq 0\}$, $K_-=\Span(\frac\p{\p x_3}), K_+=\Span(\frac\p{\p x_2})$;
\item {\sl Index $1$ membrane with negative boundary}.
$V_1^--=\{x_2,x_3=0, x_1\leq 0\}$, $K_-=\Span(\frac\p{\p x_2}), K_+=\Span(\frac\p{\p x_3})$
\item {\sl Index $2$ membrane with positive boundary},
$V_2^+=\{x_1,x_2=0, x_3\geq 0\}$, $K_+=\{0\}, K_-=\Span(\frac\p{\p x_1},\frac{\p}{\p x_2})$.

 \end{enumerate}
  In  the  cases (i)-(ii) the fold has index 1  and in the cases (iii)-(iv) index 2 with respect to the outward orientation of the boundary of the projection $\oV$  of $V$.
 \medskip
 
A \emph{$k$-nodal fibration} $f:Y\to Z$ is a fiber bundle whose fibers
are $k$-nodal surfaces, equipped with $k$ disjoint fiberwise
embeddings $\psi_i: Z \times K_0 \to Y$ over $Y$, such that the
complement of the images of the $\psi_i$ forms a smooth fiber bundle
over $Z$.

Let $f:Y\to Z$ be a $k$-nodal fibration. Denote $\wh Z:=Z\times I^k$
and construct a manifold $\wh Y$ together with a map $\wh f:\wh Y\to
\wh Z$ as follows.  Set
$$\wh Y= \left(Y\setminus\bigcup\limits_1^k\psi_i( Z\times
 K_0)\right)\times I^k\mathop{\cup} \limits_{\sigma_1} (Z\times H\times
I^{k-1})\mathop{\cup}\limits_{\sigma_2}
\dots\mathop{\cup}\limits_{\sigma_k}(Z\times H\times I^{k-1}),\;$$
where $\sigma _i:Z\times(\p_+H\cup\p_-H)\times I^{k-1}\to \psi_i(Z\times\p K_0)\times
I^k$, $i=1,\dots, k$, are gluing diffeomorphisms defined by the
formula 
$$\sigma_i(z,x,t_1,\dots, t_{k-1})=\psi_i(z,\beta^S_\pm(x), t_1,\dots,
t_{i-1},\beta^I_\pm(x),t_i,\dots, t_{k-1})$$ for $x\in\p_\pm H$,
$z\in Z$ and $t_j\in I$ for $j=1,\dots,k-1$.  The map $f:Y\to X$
extends to a map $\wh f:\wh Y\to \wh Z$ as equal to the projection
$(y,t)\mapsto(f(y),t)\in \wh Z=Z\times I^k$ for $(y,t)\in
\left(Y\setminus\bigcup\limits_1^k\psi_i( K_0\times Z)\right)\times
I^k$ and equal to the map $$(z,x,t_1,\dots,
t_{k-1})\mapsto(z,t_1,\dots,t_{i-1},Q(x),t_i,\dots, t_k)$$ on the
$i$-th copy of $Z\times H\times I^{k-1}$ glued with the attaching map
$\sigma_i$.  Note that the map $\wh f$ has $k$ fold components which
are mapped to the hypersurfaces $C_i=\{t_i=0\}\subset\wh Z=Z\times
I^k$.  Thus $Z=Z\times 0=\bigcap\limits_1^k \oC_i$ is the locus of
$k$-multiple intersection of images of fold components of $\wh f$.  We
will call the folded map $\wh f:\wh Y\to \wh Z$ the {\it universal unfolding} of the
$k$-nodal fibration $f:Y\to Z$. 
\medskip

The following lemma, which   follows from the local description of an enriched  folded map   in a neighborhood of its fold, see Lemma
\ref{lemma-replacing-eq:14},
 shows that the universal unfolding describes the structure of a folded map over a neighborhood of the locus of maximal multiplicity of fold intersection.
 
 \begin{lemma}\label{lm:k-multiple} 
Let $ f:M\to X $ be a folded map with cooriented hyperbolic folds.   Suppose that all combinations of (projections of) fold components
intersect transversally among themselves and with $\p X$.  Let $k$ be
the maximal multiplicity of the fold intersection and $Z $ be one of the components of the $k$-multiple intersection. Then
$Z\subset X$ is a submanifold with boundary $\p Z\subset\p X$, and the
restriction $f|_{Y=f^{-1}(Z)}:Y\to Z$ is a $k$-nodal fibration. 
   Let $\wh f :\wh Y\to\wh Z=Z\times I^k$ be the universal unfolding of
  the $k$-nodal fibration $f|_Z$. Then there exist embeddings $\phi
  :\wh Z\to X$ and $\Phi :\wh Y\to M$ which extend the inclusions
  $Z\hookrightarrow X$ and $Y\hookrightarrow M$ such that the diagram
  \begin{equation}
    \xymatrix{\wh Y\ar[r]^{\Phi}\ar[d]^{\wh f }&M\ar[d]^f\\
      \wh Z  \ar[r]^{\phi} &X}
  \end{equation}
  commutes.
  If the folded map is enriched then, depending on indices of the membranes adjacent to the intersecting folds, and signs of the folds as boundary components of the membranes, one can arrange that the pre-images of the membranes and their framings under the embedding  
$\Phi:\wh Y\to M$ coincide with the submanifolds
$Z\times V_j^\pm\times I^{k-1}$, $j=0,1,2$, and their defined above  model framings in the corresponding copies of $Z\times H\times I^{k-1}$ in $\wh Y$.
\end{lemma}

\medskip

Notice that Theorem \ref{thm:Harer-geometric} is easily generalized to $k$-nodal
fibrations as follows.

\begin{theorem}[Geometric form of Harer stability for nodal fibrations]
  \label{thm:Harer-geometric-nodal}
  Let $\Sigma_1 \subset \Sigma_2$ be compact surfaces with boundary
  (not necessarily connected).  Let $f: M \to X$ be a $k$-nodal
  fibration with $T_\infty$ ends, and let
  \begin{align*}
    j: (\partial X \times \Sigma_2) \cup (X \times \Sigma_1) \to M
  \end{align*}
  be a fiberwise embedding over $X$, such that its image in each fiber
  is disjoint from the nodes, and that in each fiber the complement of
  its image is connected, even after removing all nodes.

  Then, after possibly changing $f: M \to X$ by a bordism which is the
  trivial bordism over $\partial X$, the embedding $j$ extends to an
  embedding of $X \times \Sigma_2$, still disjoint from nodes and with
  connected complement.
\end{theorem}
   
\subsection{Proof of Theorem \ref{thm:Harer2}}\label{sec:proof-Harer}

Let $(f:M\to X,\fe)$ be an enriched folded map from $\Fold^{\$}_h$.
Let $U\subset X$ be a compact domain with smooth boundary. We assume
that all combinations of (projections of) fold components intersect
transversally among themselves and with $\p U$.  Let $k$ be the
maximal multiplicity of the fold intersection. We denote by $U_j$,
$j=1,\dots, k$, the set of intersection points of multiplicity $\geq
j$ in $U$, and set $U_0:=U$.  Thus we get a stratification
$U=\bigcup\limits_0^k U_j\setminus U_{j+1}$, and we have $ U_j
=\bigcup\limits_{i\geq j} U_i$. Set $M_j=f^{-1}(U_j)$.  Note that
$U_k$ is a closed submanifold of $U$ with boundary $\p U_k\subset\p U$.  The
map $f_k=f|_{M_k}:M_k\to U_k$ is a $k$-nodal fibration. The membranes
which are not adjacent to the fold components intersecting along
$U_k$, together with their framings define fiberwise embeddings $
s_1,\dots, s_l:U_k\times D^2\to M_k$ over $U_k$, disjoint from the
image of $j$ and from each other.

Let us apply Theorem \ref{thm:Harer-geometric-nodal}, the nodal
version of Harer's stability theorem, to the nodal fibration $ f_k $
and the fiberwise embedding
\begin{align*}
  \wt j=j\sqcup\bigcup\limits_1^l s_i: (\p
  U_k\times\wt\Sigma_2)\cup(U_k\times \wt\Sigma_1)\to M_k,
\end{align*}
where $\wt\Sigma_1$ and $\wt\Sigma_2$  are  disjoint unions of $\Sigma_1$ and $\Sigma_2$, respectively, with $l$ copies
of the disc $D^2$.   As a result, we
find a bordism $ F_k: W_k\to Y_k$ (in the class of $k$-nodal
fibrations) between the $k$-nodal fibrations $ f_k : M_k \to U_k$ and
$ f'_k: M'_k\to U'_k$ which is constant over $\Op\p U_k$ and such that
the fiberwise embedding $\wt j$, extends to a fiberwise embedding
$$ J:\left((\p' Y_k= \p
  U_k\times[0,1])\cup(\p_+Y_k= U_k')\right)\times
\wt\Sigma_2 \to W_k.
$$
     
Let $\wh F_k:\wh W_k\to\wh Y_k = Y_k \times I^k$ be the universal
unfolding of the $k$-nodal fibration $ F_k:\ W_k\to Y_k$.  We view
$\wh F_k$ as a bordism between the universal unfoldings $\wh f_k:\wh
M_k\to\wh U_k = U_k \times I^k$ and $\wh f'_k:\wh M_k\to\wh U'_k=
U'_k\times I^k$ of the $k$-nodal fibrations $f_k$ and $f_k'$. The
embeddings $J $ extends to a fiberwise embedding
\begin{align*}
  \wh J:\left((\wh{\p' Y}_k=\p' Y_k\times I^k)\cup \wh U_k'
  \right)\times \wt\Sigma_2 \to \wh W_k.
\end{align*} 
over $\wh Y_k$.  According to Lemma \ref{lm:k-multiple} the
restriction of $f$ to a tubular neighborhood of $U_k$ is isomorphic to
the universal unfolding $\wh f_k:\wh M_k\to\wh U_k$ of the $k$-nodal
fibration $f_k$. In other words, there exist embeddings $\phi_k:\wh
U_k\to X$ and $\Phi_k:\wh M_k\to M$ which extend the inclusions
$U_k\hookrightarrow X$ and $M_k\hookrightarrow M$ such that the
diagram
\begin{equation}
  \xymatrix{\wh M_k\ar[r]^{\Phi_k}\ar[d]^{\wh f_k}&M\ar[d]^{f}\\
    \wh X_k \ar[r]^{\phi_k} &X}
\end{equation}
commutes.
 Moreover, $(\phi_k,\Phi_k)$ can be chosen in such a way that the framed membranes
 adjacent to intersecting folds correspond to model framed membranes of Lemma  \ref{lm:k-multiple}.
 
Let us glue the bordism $\wh F_k:\wh W_k\to\wh Y_k$ to the trivial
bordism $ F=f\times\Id:W=M\times I\to Y=X\times I$ using the attaching
maps $(\phi_k, \Phi_k)$:
\begin{align*}
  \wt F: \wh W_k\mathop{\cup}\limits_{\Phi_k\times 1} M\times I\to \wh
  Y_k\mathop{\cup}\limits_{\phi_k\times 1} X\times I
\end{align*}
and then smooth the corners.  The folded map $\wt F:\wt W\to\wt Y$
resulting from this construction is a bordism between $f:M\to X$ and
$f':M'\to X'$, which is trivial over the complement of $\phi_k(\wh
U_k)\subset X$. Model framed membranes  of Lemma  \ref{lm:k-multiple}
provide  us with a canonical extension to $\wh W_k$ of framed  membranes
adjacent to $Y_k$.  On the other hand,  the restriction of $\wh J$ to $\wh Y_k \times
\coprod^l D^2 \subset \wh Y_k \times \Sigma_2$ allows    us  to extend  
all the other  framed membranes.
 Thus the constructed map $\wt F:\wt W\to\wt Y$ together with the
extended enrichment is a bordism in the category $\Fold^{\$}_h$.  The
fiberwise embedding $ \wh J :U_k'\times\Sigma_2\to M_k'$ extends to a
closed neighborhood $\Omega\supset U_k'$ in $M'$.  Consider the domain
$U''=U'\setminus\Int \Omega$. The maximal multiplicity of fold
intersection over $U''$ is equal to $k-1$. Hence, we can repeat the
previous argument to extend $\wh J$ over the stratum $U''_{k-1}$,
possibly after changing it by another bordism in the category
$\Fold^{\$}_h$. Continuing inductively we find the required extension
to the whole domain bounded by $\p U$. \qed
  
\section{Proof of Theorem \ref{thm:main1}}\label{sec:main-proof}
As it was already mentioned above, Theorem \ref{thm:main1} follows from
Theorem \ref{thm:epi-folds} and Propositions
\ref{prop:enriched-hyperbolic} and \ref{prop:hyperbolic-fibrations}.
We prove these propositions in the next two sections.

\subsection{Proof of Proposition \ref{prop:enriched-hyperbolic}}
\label{sec:to-hyperbolic}
 
Let $(f,\fe)$ be an enriched folded map with $f: M \to X$.  First, we
get rid of non-hyperbolic folds. Let $Z$ be a non-hyperbolic fold
component.   The following procedure, which is illustrated in Fig. \ref{mw11}, replaces $Z$ by a parallel hyperbolic fold.

Let  $V$ be the membrane of $Z$, and $N=\oZ\times [-2,0] \subset X$ be an interior collar of
$\oZ=\oZ\times 0$ in $X\setminus\Int \oV$.  Let us recall that the map $f$ has a standard
end, where it is equivalent to the trivial fibration $\Tinf\times X\to
X$.  Let $\overline A= \oZ\times[-2,-1]\subset N$ and $\overline B= \oZ\times[-2,-1]\subset N$, so that $N=\overline A\cup\overline B$.
Let us  lift $\overline A$ to   an annulus   $A=z\times A\subset M$, $z\in
\Tinf$.  Write $\overline Z_1:=Z\times (-1),$ $\overline
Z_2:=Z\times(-2 )$, $Z_1:=z\times \overline Z_1$,
$Z_2:=z\times\overline Z_2$, so that $\p\overline A=\overline
Z_1\cup\overline Z_2$ and $\p A=Z_1\cup Z_2$.  Using Lemma
\ref{lm:2fold-creation} we can create a double fold with the membrane
$A$ and folds $Z_2$ of index $1$ and $Z_1$ of index $0$ with respect
to the coorientation of $\overline Z_1,\overline Z_2$ by the second
coordinate of the splitting $N=Z\times [-2,0]$. The  folds $Z_1$ and $Z$ have index
$0$ with respect to the outward coorientation of $\p \overline B$, and hence we
can use a membrane expanding surgery to kill both folds, $Z$ and
$Z_1$, and spread their membranes over $\overline B$.  As a result of this procedure we have  replaced $Z$ by a hyperbolic fold $Z_2$.

\begin{figure}[hi]
\centerline{\includegraphics[height=80mm]{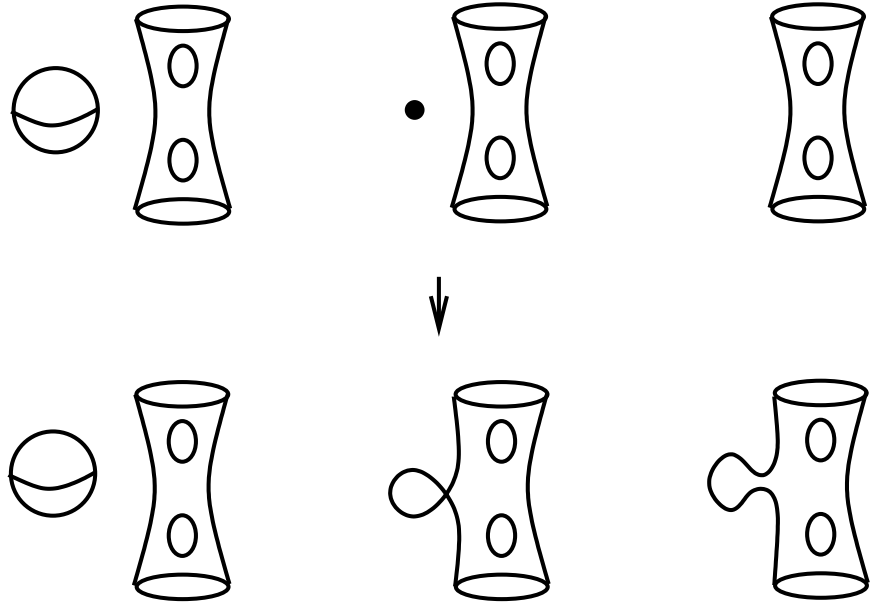}}
\caption{Replacing an elliptic fold by a hyperbolic one}
\label{mw11}
\end{figure}

\medskip It remains to make the fibers of $f|_{M\setminus\Sigma(f)}$
connected.  We begin with the following lemma from \cite{MW07}.  \mn
\begin{lemma}\label{lm:discs}
  Let $M\to X$ be an enriched folded map without folds of index $0$.
  Then there exist disjoint $n$-discs $D_i$, $i=1,\dots, K$, embedded
  into $M$, such that
  \begin{itemize}
  \item $f|_{D_i}$ is embedding $D_i\to \Int X$ for each $i=1,\dots
    K$;
  \item $V\cap\bigcup\limits_1^K D_i =\varnothing$, where $V\subset M$
    is the union of all membranes;
  \item for each $x\in X$ each irreducible component of $\pi^{-1}(x)$
    intersects at least one of the discs $D_i$ at an interior point.
  \end{itemize}
\end{lemma}

\begin{proof}
  Note first, that the statement is evident for any {\it fixed\, }
  $x\in X$. Hence, without controlling the disjointness of the disks
  $D_i$ the statement just follows from the compactness of $X$. One
  can choose the required {\it disjoint} disks $D_i$ using the
  following trick: fix a function $h:M\to \bbR$ and take the disks
  $D_i$ such that each disk belongs to its own level hypersurface of
  $h$.  When we choose such disks for $x\in X$ one needs to avoid the
  points $z\in F_x=f^{-1}(x)$ where the level hypersurface is tangent
  to the fiber $F_x$.  It can be done by a small perturbation of
  disks, if the complement of all ``bad'' points (for all $x\in X$) is
  {\it open} and {\it dense} in $F_x$ for all $x\in X$.  But Thom's   jet
  transversality theorem asserts that this is a generic situation for  functions
  $h:M\to \bbR$.
\end{proof}

Let the disks $D_i \subset M$ be as in Lemma \ref{lm:discs}.  Let us also
consider discs $\Delta_i=\overline D_i\times y_i$, $i=1,\dots, K$,
where $y_1,\dots y_K$ are disjoint points at the end $\oTinf$ of the
fiber $F$. Next, using each pair $(D_i,\Delta_i)$, $i=1,\dots, K$, as
a basis for an index 1 fold creating surgery we create new folds
$\Sigma_i=\p\wt D_i$ of index $1$ with the discs $\overline D_i$
serving as membranes, while making all the fibers connected, see Remark \ref{rem: fold-creating}.

\medskip

As a result of this step we arrange all fibers $F_x=
f^{-1}(x)\setminus\Sigma(f), x\in X$, to be connected.  This completes
the proof of Proposition \ref{prop:enriched-hyperbolic}.

\qed

\subsection{Proof of Proposition \ref{prop:hyperbolic-fibrations}}
\label{sec:to-fibrations}

Let $V_1, \dots, V_N$ be the collection of (connected) framed membranes of
$\ff=(f,\fe)$. We are going to inductively remove all of them.  If the
membrane $V_1$ is pure then, in view of Corollary \ref{cor:basis-existence},
we can assume, after a possible change of $f$ by a bordism in the
category $\Fold^{\$}_h$, that there is a basis for the fold surgery
which removes $\p V_1$ together with the membrane.
  
Suppose now that the membrane $V_1$ is mixed.  Consider first the case
$n>1$. By our assumption  each boundary component  $\C$ of $\oV_1$ bounds in this case a domain $U_C\subset X$. If $\p X\neq \varnothing$ then the domain $U_C$ is uniquely defined. If $X$ is closed then we fix a point $p\in X\setminus \oV_1$ and denote by $U_C$ the domain which does not contain $p$.
There exists exactly one  boundary component  $C$   of $V_1$ such that $\oV_1\subset U_C$.  For any other boundary component $C'$ of $V_1$ we have $U_{C'}\subset X\setminus\Int\oV_1$. We will refer to $C$ as  the exterior fold of $V_1$, and to all other boundary folds of $V_1$ as its interior folds.
First, we apply 
  Corollary \ref{cor:basis-existence} in
order to create a basis for a surgery which eliminates each interior fold $C'$ and spreads
the membrane $V_1$ over $U_{C'}$.  After applying this procedure to all interior folds 
 of $V_1$   we
will come to the situation when $\p V_1=C$ is connected, and hence the
membrane $V_1$ is pure, which is the case we already considered
above. Applying the same procedures subsequently to the membranes
$V_2,\dots, V_N$ we complete the proof of Proposition
\ref{prop:hyperbolic-fibrations} when $n>1$.

\medskip
Finally consider the case $n=1$, i.e. $X=I$ or $X=S^1$.
We assume for determinacy that $X=I$, i.e. $f$ is a Morse function. A mixed framed membrane    of $f$ connects two critical points $p_1,p_2$ of
$f$ of index $1$ and $2$, and with critical values $c_1,c_2, c_1<c_2$,
respectively. For a small $\eps>0$ let us consider vanishing circles
$S_1\subset F_{c_1+\epsilon}, S_2\subset F_{c_2-\epsilon}$ of critical
points $p_1$ and $p_2$, where we denote $F_t:=f^{-1}(t),\;
t\in\R$. Choose an oriented embedded circle $S_1'\subset
F_{c_1+\epsilon}$ which intersects $S_1$ transversally in one point. Let
$A$ denote the annulus $S^1\times [-1,1]$, and $D\subset \Int A$ be a
small 2-disc centered at a point $q\in S^1\times 0$.  Choose a
fiberwise embedding $k:I_\eps\times D\to M_\eps=f^{-1}(I_\eps)$ over
$I_\eps=[1+\eps,2-\eps]$, such that $k(I_\eps\times 0)\subset V$ and
the linearization of $k$ along $k(I_\eps\times 0)$ provides the given
framing of the membrane $V$.  Let us also choose embeddings $j_1:A\to
F_{1+\eps}$ and $j_2: A\to F_{2-\eps}$ such that $j_1|_{S^1\times
  0}=S_1'$, $j_2|_{S^1\times 0}=S_2$, and for $x\in D$ we have
$j_1(x)=k(1+\eps,x), j_2(x)=k(2-\eps,x)$. Let us apply Theorem
\ref{thm:Harer2} to the Morse function ($=$ the folded map)
$f'=f_{M_\eps}:M_\eps \to I_\eps$ with $\Sigma_2=A $, $\Sigma_1=D\subset\Sigma_1$ and  the embedding $j:(\p I_\eps\times \Sigma_2)\cup(I_\eps\times\Sigma_1)\to M$, which is equal to $ 
j_1\sqcup j_2$ on $\p I_\eps\times A$ and to $k$ on $I_\eps\times D$.
   Theorem \ref{thm:Harer2}
then allows us, after possibly changing $f$ in its bordism class in
$\Fold_h^{\$}$, to extend $j_1\sqcup j_2$ to a fiberwise over $I_\eps$
map $j:I_\eps\times A\to M$ which coincides with $k$ on $I_\eps\times
D$. Choosing a metric for which the cylinder $f(I_\eps\times
(S^1\times0))$ is foliated by gradient trajectories, one of which is
$V$, we come to the situation when we can apply the standard Morse
theory cancellation lemma, see for instance \cite{Milnor}, to kill
both critical points $p_0$ and $p_1$ together with their membrane
$V$. The cancellation deformation is inverse to the double fold
creation, and hence it can be realized by a bordism in the category
$\wt\Fold_h^{\$}$.  \qed
    
\medskip This completes the proof of Theorem \ref{thm:main1}.
\medskip

 \section{Miscellanous}\label{sec:misc}
\subsection{Appendix A: From wrinkles to double folds}
\label{sec:appendix-a:-from}

\subsubsection{Cusps}\label{sec:cusps}
\mn
Let $n>1$. Given a map $f:M\to X$, a point $p \in \Sigma(f)$ is called a
{\it cusp} type singularity or a {\it cusp}
of index $s + \frac{1 }{ 2}$\; if near the point $p$ the map $f$ is
equivalent to the map
$$
\bbR ^{n-1}\times \bbR ^1\times \bbR ^{d}
\rightarrow \bbR^{n-1} \times \bbR ^1
$$
given by the formula
\begin{equation}\label{eq:cusp}
(y,z,x) \mapsto \left( y,z^3 + 3y_1z - \sum^s_1 x_i^2 +
\sum^{d}_{s+1} x^2_j \right)  
\end{equation}
where $x=(x_1, \ldots, x_{d}) \in \bbR ^{d},\,\, z \in \bbR ^1,\,\,
y = (y_1, \ldots, y_{n-1}) \in \bbR ^{n-1}$.

The set of cusp points is denoted by $\Sigma^{11}(f)$. It is a codimension $1$ submanifold of $\Sigma(f)$ which is in the above canonical coordinates is given by $x=(x_1,\dots,x_d)=0, y_1=z=0.$
The vector  field $\frac{\p}{\p y_1}$ along $\Sigma^{11}(f)$ is called the 
{\it characteristic vector} field of the cusp locus. 
It can be invariantly  defined   as follows.
Note that for any point $p\in\Sigma^{11}(f)$ there exists  a neighborhood $\Omega\ni f(p)$ in $X$ such that
$\Omega\cap f(\Sigma(f))$ can be presented as a union of two manifolds  $\overline\Sigma_\pm$ with the common boundary $\p\overline\Sigma_\pm=\Omega\cap f(\Sigma^{11}(f))$, the common tangent space $T=T_{f(p)}\overline\Sigma_\pm=df(T_pM)$ at the point $f(p)$, and the common outward coorientation $\nu$ of  $T'=T_p{\p\overline\Sigma_\pm}\subset T $.  On the other hand, the differential $df$ defines an isomorphism $$T_pM/(\Ker\,d_pf+T_p \Sigma(f))\to T/ T'.$$
Hence, there exists a  vector field $Y$ transversal to
 $\Ker\,df+T \Sigma(f)$ in $TM$ along $\Sigma^{11}(f)$, 
 whose projection defines the coorientation $\nu$ of  $T'$ in $T$ for all points $p\in\Sigma^{11}(f)$.  
 One can show that any vector field 
 $Y$ defined  that way coincides with the vector field
 $\frac{\p}{\p y_1}$ for some  local coordinate system in which the map $f$ has the canonical form \eqref{eq:cusp}.
 
 Note that the line bundle $\lambda$ over $\Sigma^{11}(f)$ is always trivial.
 Indeed, $\lambda$ can be equivalently defined as the kernel of the quadratic form $d^2f:\Ker\,df\to\Coker\,df$, and thus one has an invariantly defined cubic form $d^3f:\lambda\to\Coker\,df$ which does not vanish. The bundle $\Ker\,df|_{\Sigma^11(f)}$ can be split  as $\Ker_+\oplus\Ker_-\oplus\lambda$, so that  the quadratic form $d^2f$ is positive definite on $\Ker_+$ and negative definite on $\Ker_-$.
 
  \subsubsection{Wrinkles and wrinkled mappings}
\label{ss:wrinkles}

Consider the map
$$
w(n+d,n,s): \bbR^{n-1}  \times \bbR^1\times \bbR^{d}
\rightarrow \bbR^{n-1} \times \bbR^1
$$
given by the formula
$$
(y,z,x) \mapsto
\left( y, z^3 + 3(|y|^2-1)z - \sum^s_1 x^2_i + \sum_{s+1}^{d}
x_j^2 \right)\,,
$$
where $y \in \bbR ^{n-1},\,\, z \in \bbR^1 ,\,\, x\in \bbR ^{d}$ and
$|y|^2 = {\sum_1^{n-1}y^2_{i}}$.

\mn
The singularity
$\Sigma (w(n+d,n,s))$ is the $(n-1)$-dimensional sphere
$$
S^{n-1}=S^{n-1}\times 0\subset\bbR^n\times\bbR ^{d}.
$$
whose equator $\Sigma^{11}(f)= \{ |y|=1, z=0, x=0\} \subset \Sigma (w(n+d,n,s)) $
consists of cusp points of index $s + \frac{1 }{ 2}$\,.  The upper
hemisphere $\Sigma (w) \cap\{z>0\}$ consists of folds of index $s$,
while the lower one $\Sigma (w) \cap \{z< 0\}$ consists of folds of
index $s+1$. The radial vector field $Y=\sum\limits_1^{n-1}y_j\frac{\p}{\p y_j}$ serves as a characteristic vector field  of the cusp locus.

\begin{figure}[hi]
\centerline{\includegraphics[height=40mm]{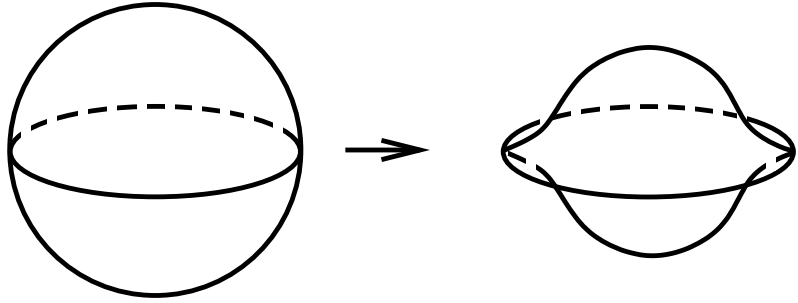}}
\caption{Wrinkle in the source and in the image}
\label{mw9}
\end{figure}

\mn
Although the differential
$dw(n+d,n,s):T(\bbR^{n+d}) \rightarrow T(\bbR^n)$
degenerates at points of $\Sigma(w)$, it can be canonically {\it regularized}
over $\Op_{\bbR^{n+d}}D^n$, an open neighborhood of the disk
$D^n=D^n\times 0\subset \bbR^n\times\bbR^{d}$.
Namely, we can substitute the element $3(z^2 + |y|^2 -1)$ in the Jacobi matrix of
$w(n+d,n,s)$ by a function $\gamma$ which
coincides with $3(z^2+|y|^2-1)$ on $\bbR^{n+d}\setminus\Op_{\bbR^{n+d}}D^n$ and does
not vanish along the $n$-dimensional subspace
$\{x=0\}=\bbR^{n}\times {\bf 0}\subset \bbR^{n+d}\,.$
The new bundle map
${\cal R}(dw):T (\bbR ^{n+d}) \rightarrow T(\bbR ^n)$
provides a homotopically canonical
extension of the map
$dw:T(\bbR ^{n+d}\setminus \Op_{\bbR^{n+d}}D^n) \rightarrow T(\bbR ^n)$ to
an epimorphism (fiberwise surjective bundle map)
$T(\bbR^{n+d}) \rightarrow T(\bbR ^n)$.  We call
${\cal R}(dw)$ the {\it regularized differential} of the map $w(n+d,n,s)$.

\mn
A map $f:U\rightarrow X$ defined on an open ball $U \subset M$ is called a
{\it wrinkle} of index $s +\frac {1 }{ 2}$\,\, if it is equivalent to the
restriction  $w(n+d,n,s)|_{\Op_{\bbR^{n+d}}D^n}$.
We will use the term
``wrinkle'' also for the singularity $\Sigma (f)$ of a wrinkle $f$.

\mn
Notice that for $n=1$ the wrinkle is a function with two nondegenerate
critical points of indices $s$ and $s+1$ given in a neighborhood of a gradient
trajectory which connects the two points.

\mn
A map $f:M \rightarrow X$ is called {\it wrinkled} if there exist disjoint
open subsets $U_1, \ldots, U_l \subset M$ such that the restriction
$f|_{M\setminus  U}, \,\, U=\bigcup^l_1 U_i, $ is a
submersion (i.e. has rank equal $n$) and for each $i=1, \ldots, l$ the
restriction $f|_{U_i}$ is a wrinkle.

\mn
The singular locus $\Sigma(f)$ of a wrinkled map $f$ is a union of
$(n-1)$-dimensional spheres (wrinkles) $S_i=\Sigma(f|_{U_i})
\subset U_i$.
Each $S_i$ has a $(n-2)$-dimensional equator $S'_i \subset S_i$
of cusps which divides $S_i$ into two hemispheres of folds of two neighboring
indices. The differential $df:T(M) \rightarrow T(X)$ can be regularized to
obtain an epimorphism
${\cal R} (df):T(M) \rightarrow T(X)$.  To get
${\cal R} (df)$ we regularize $df|_{U_i}$ for each wrinkle $f|_{U_i}$.

\mn
The following  Theorem \ref{thm:wrinkled}
is the main result of the paper [EM1]:

\begin{theorem}
[Wrinkled mappings]\label{thm:wrinkled}
Let $F: T(M) \rightarrow T(X)$ be an
epimorphism which covers a map $f: M\rightarrow X$.
Suppose that $f$ is a submersion on a neighborhood of a
closed subset $K \subset M$, and $F$
coincides with $df$ over that neighborhood.
Then there exists a wrinkled map $g : M \rightarrow X$ which coincides with
$f$ near $K$ and such that ${\cal R} (dg)$ and $F$ are homotopic
rel. $T(M)|_K$. Moreover, the map $g$ can be chosen arbitrarily $C^0$-close to $f$
and with arbitrarily small wrinkles.
\end{theorem}

\subsubsection{Cusp eliminating surgery}\label{sec:cusp-surgery}

We are going to modify each wrinkle to a spherical double fold
using {\it cusp elimination surgery}, which is  one of the surgery operations studied in \cite{El72}. Unlike fold elimination surgeries described above in Section \ref{sec:surgery-general}  cusp elimination surgery does not affect the underlying manifold and changes a map by a homotopic one.  For maps
$\bbR^2\to\bbR^2$ the operation  is shown on Fig.\ref{mw1}.
  
\begin{figure}[hi]
\centerline{\includegraphics[height=70mm]{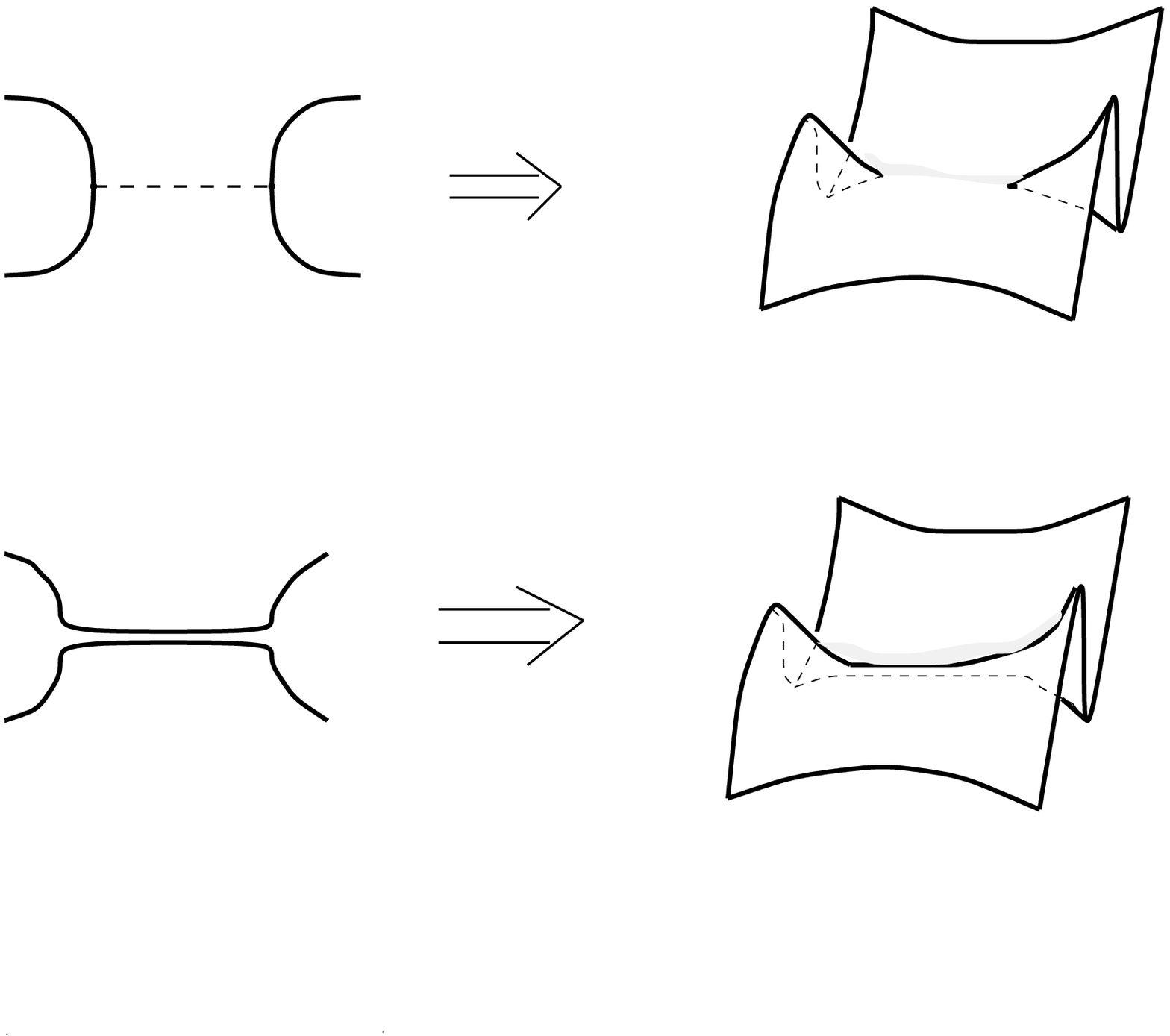}}
\caption{Cusp eliminating surgery in the case $n=2, d=0$}
\label{mw1}
\end{figure}
\begin{definition}{\rm
Let $C\subset\Sigma(f)$ be a connected component of the cusp locus. Let $Y$ be the characteristic vector field of $C$. Suppose that the bundles $\Ker_-,\Ker_+$ and $\lambda$ over $C$ are trivialized, respectively, by the frames 
$$\left(\frac{\p}{\p x_1},\dots, \frac{\p}{\p x_s}\right),\; \left(\frac{\p}{\p x_{s+1}},\dots, \frac{\p}{\p x_d}\right)\;\hbox{and}\; \frac{\p}{\p z}.$$
A {\em basis } for a cusp eliminating surgery consists
 of an $(n-1)$-dimensional submanifold $A\subset M$ bounded by $C$, together with an extension of the above framing as a trivialization of the normal bundle   $\nu$ of $A$ in $M$, such that
 \begin{itemize}
 \item $ f|_{\Int A}:\Int A\to X$ is an immersion;
  \item the characteristic vector field $Y$ is tangent to $A$ along $C$, and inward transversal to
 $C=\p A$;
 \item $\frac{\p}{\p x_j}\in\Ker\,df$ for all $j=1,\dots, d$.
 \end{itemize}}
 \end{definition}

Let us extend $A$ to a slightly bigger manifold $\wt A$ such that $\Int \wt A\supset A$, and extend the framing over   $\wt A$. 
One can show (see \cite{El72} and \cite{Ar76}) that there exists a splitting
 $U\to\wt A\times\R \times \R^d $ of a tubular neighborhood of $\wt A$ in $M$, such that in the corresponding local coordinates $ y\in\wt A,\, z\in\R$ and $x=(x_1,\dots, x_d)\in\R^d$ the map $f$
 can be presented as a composition $$U\mathop{\to}\limits^F\wt A\times\R\mathop{\to}\limits^h X,$$
 where $h$ is an immersion and $F$ has the form
 $$F(  y,z, x)=\left( y, z^3+\phi_0(\wt y)\sigma(\frac1\epsilon(z^2+\sum\limits_1^d x_j^2))z- \sum^s_1 x_i^2 +
\sum^{d}_{s+1} x^2_j\right) ,$$ where the function $\phi_0:\R\to\R$ satisfies
 $\phi_0<0 $ on $\Int A\subset \wt A$ and $\phi_0>0 $ on $\wt A\setminus A$,   $\sigma:[0,1]\to[0,1]$ is a cut-off function equal to $1$ near $0$ and to $0$ near $1$, and $\eps>0$ is small enough.
 
 \medskip
 Consider another function $\psi_1:\wt A\to(-\infty,0)$ which coincides with $\phi_0$ outside $\Op A\subset\wt A$ and such that $|\phi_1|\leq|\phi_0|.$
 Denote $\phi_t:=(1-t)\phi_0+t\phi_1$, $t\in[0,1]$, and consider homotopies
 $$F_t(y,z,x)=\left(\wt y, z^3+\phi_t(\wt y)\sigma(\frac1\epsilon(z^2+\sum\limits_1^d x_j^2))z- \sum^s_1 x_i^2 +
\sum^{d}_{s+1} x^2_j\right),$$ $( y,z,x)\in U$,
and $f_t=h\circ F_t:U\to X$. The homotopy $f_t$ is supported in $U$ and hence can be extended to the whole manifold $M$ as equal to $f$ on $M\setminus U$. The next proposition is straightforward.
 \begin{proposition}\label{prop:cusp-surgery}
 \begin{enumerate}
 \item The homotopy $f_t$ removes the cusp component $C$. The map $ f_1$ coincides with $f_0$ outside $U$, has only fold type singularities in $U$, and   $$\Sigma(f_1|_U)=\{x=0, z^2=-\wt\phi_1(\wt y)\}.$$
 \item   Suppose that $\Sigma(f)\setminus C$ consists of only fold points and that the restriction of the map $f$ to $\Sigma\cup A$ is an embedding. Then the restriction $f_1|_{\Sigma (f_1)}:\Sigma (f_1)\to X$ is an embedding provided that the neighborhood $U\supset A$ in the surgery  construction is chosen small enough.
 \end{enumerate}
 \end{proposition}
 
 \subsubsection{From wrinkles to double folds}\label{sec:wrinkle-double}

 \begin{proposition}
 \label{prop:preparation}
Let
$$w(n+d,n,s): \bbR^{n-1}  \times \bbR^1\times \bbR^{d}
\rightarrow \bbR^{n-1} \times \bbR^1$$
be the standard
wrinkled map with the wrinkle $S^{n-1}\subset \bbR^n\times \bbR^d$. Suppose that $n>1$. Then
\begin{description}
\item{a)} there exists an embedding
$$h:D^{n-1}\to \Op_{\bbR^{n+d}}D^n$$ and a framing $\mu$ of the normal bundle to $A=h(D^{n-1} )\subset  \bbR^n\times \bbR^d$ such that the pair  $(A,\mu)$ forms a basis for a surgery eliminating the cusp $\Sigma^11(w)=S^{n-2}\subset S^{n-1}$ of the wrinkle;
\item{b)} if $d>0$ then one can arrange that the map $w(n+d,n,s)$ restricted to $\Sigma(w(n+d,n,s))\cup A$ is an embedding.
\end{description}
 \end{proposition}
 
\begin{proof}  

It is easy to construct an embedding $h_0$ and a framing $\mu $  to satisfy a). 
\begin{figure}
\centerline{\includegraphics[height=30mm]{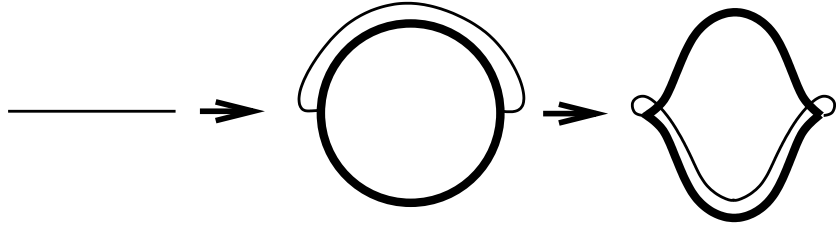}}
\caption{The embeddings $h_0$ and $g_0=w\circ h_0 $ (thin lines)}
\label{mw2}
\end{figure}
The construction is clear from  Fig.\ref{mw2}. The manifold $A$ in this case is obtained from the boundary  of the upper semi-ball $\{|y|^2+z^2\leq 1+\delta,z\geq 0\}\subset\R^{n-1}\times\R$ by removing the open disc
$D^{n-1}=\{z=0,|y|<1\}$, and then smoothing the corner. Here $\delta>0$ should be chosen small enough so that $A$ lie in the prescribed neighborhood of the wrinkle.

The framing $\mu$ is given by $\frac{\p}{\p x_1},\dots,\frac{\p }{\p x_d}$ and the normal vector field to $A$ in $\R^{n-1}\times\R$ which coincides with $\frac{\p}{\p z}$ near $\p A$.


Unfortunately the embedding $h_0$
does not satisfies the property $b)$. However,  if $d>0$  this can be corrected as follows.
We suppose that the index 
 $s> 0$ (if $s=0$ then one should start with an embedding $h$ obtained by smoothing the boundary of the {\it lower} semi-ball). 
 Let us denote by $g_0$ the composition $w\circ h_0:D^{n-1}\to\R^{n-1}\times\R$, and by by $g^{n-1}_0$ and $g_0^1$ the projections of $g_0$ to the first and second factors, respectively.

 \begin{figure}
\centerline{\psfig{figure=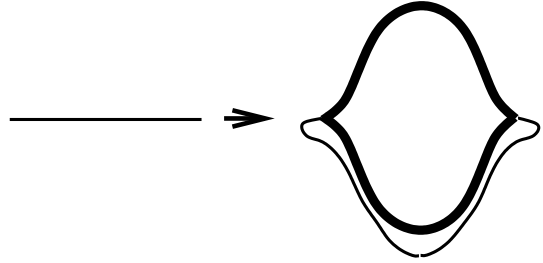,height=35mm}}
\caption{The embedding $  g$}
\label{mw4}
\end{figure}
 
For any $\eps>0$ one can choose $\delta$ in the construction of $h_0$ so small that  there exists a function $\alpha:D^{n-1}\to[0,\eps)$ such that
 \begin{itemize}
 \item $\alpha$ vanishes along $\p D^{n-1}$ together with all its derivatives;
 \item $\alpha|_{\Int D^{n-1}}>0$;
 \item the function $  g^1=g^1_0-\alpha$ has a unique interior critical point, the minimum, at $0\in D^{n-1}$;
 \item  the map $ g=(g^{n-1}_0, g^1):D^{n-1}\to\R^{n-1}\times\R$ is an embedding, and the image $ g(\Int D^{n-1})$ does not intersect the image of the wrinkle.
 \end{itemize}
  Next, take an embedding $  h: D^{n-1}\to \R^{n-1}\times\R\times \R^d$ given  by
 $$(y,z)=h_0(u), x_1=\sqrt{\alpha(u)}, x_j=0,j=2,\dots,d,$$
 $ y\in \R^{n-1}, \;z\in\R,\; x=(x_1,\dots, x_d)\in\R^d.$ Then we have $  g=w\circ h$, and hence the embedding $ h$ satisfies the property b) of Proposition \ref{prop:preparation}.
   \end{proof}
 Combining Propositions \ref{prop:preparation} and \ref{prop:cusp-surgery} we get

\begin{proposition}\label{prop:wrinkle-to-double}
 \label{thm:surgery}
   There exists a $C^0$-small perturbation of the
map $w(n+d,n,s)|_{\Op_{\bbR^{n+d}}D^n}$ in an arbitrarily small
neighborhood of the embedded disk $h(D^n)$ constructed in \ref{prop:preparation} such that the resulting map
$\widetilde w(n+d,d,s)$ is a special folded map with only one double
fold (of index $s+\frac{1}{2}$). Moreover, the regularized  differentials of $w(n+d,n,s)$ and 
$\widetilde w(n+d,d,s)$  are homotopic.
\end{proposition}

Theorem \ref{thm:wrinkled} and Proposition \ref{prop:wrinkle-to-double}
yield Theorem \ref{thm:special-folded}.
\subsection{Appendix B: Hurewicz theorem for
oriented bordism}\label{sec:Hurewitz}

As before, let $\Omega_* = \Omega_*^{\mathrm{SO}} (-)$ denote oriented
bordism.  In this appendix we give a proof of the following well known
lemma.
\begin{lemma}
  Let $f: X \to Y$ be a continuous map of topological spaces.  Then
  the following statements are equivalent.
  \begin{enumerate}[(i)]
  \item $f_*: H_k(X) \to H_k(Y)$ is an isomorphism for $k < n$ and an
    epimorphism for $k = n$.
  \item $f_*: \Omega_k(X) \to \Omega_k(Y)$ is an isomorphism for $k <
    n$ and an epimorphism for $k = n$.
  \end{enumerate}
  In particular, $f$ induces an isomorphism in homology in all degrees
  if and only if it does so in oriented bordism.
\end{lemma}

For a pair $(X,A)$ of spaces, let $\Omega_n(X,A) =
\Omega^{\mathrm{SO}}_n(X,A)$ denote the set of bordism classes of
continuous maps of pairs
\begin{equation*}
  f: (M^n, \partial M^n) \to (X,A)
\end{equation*}
for smooth oriented compact manifolds $M^n$ with boundary $\partial
M$.  There is a ``cycle map'' $\Omega_n(X,A) \to H_n(X,A)$ that maps
the class of $f$ to $f_*([M]) \in H_n(X,A)$.

Recall that $\Omega_*(X,A)$ is a ``generalized homology theory'',
i.e.\ it satifies the same formal properties as singular homology (the
Eilenberg-Steenrod axioms except the dimension axiom).  In particular
we have a long exact sequence for pairs of spaces.  If $f: X \to Y$ is
an arbitrary map, then we can replace $Y$ by the mapping cylinder of
$f$.  The axioms implies there is a natural long exact sequence
\begin{equation*}
  \dots \to \Omega_n(X) \to \Omega_n(Y) \to \Omega_n(C(f),x) \to
  \Omega_{n-1}(X) \to \dots
\end{equation*}
where $C(f)$ is the mapping cone and $x \in C(f)$ is the cone point.

It seems well known that a map $X \to Y$ induces an isomorphism in
oriented bordism if and only if it induces an isomorphism in singular
homology (with $\Z$-coefficients).  This can be seen e.g.\ from the
Atiyah-Hirzebruch spectral sequence.  We offer a more geometrical
argument, based on Hurewicz' theorem.  Recall that this says that a
map $f: X \to Y$ of simply connected spaces is a weak equivalence if
and only if it is a homology equivalence.

We will use the \emph{unreduced suspension} $\Sigma X$ of a space $X$.
This is the union $CX \cup_X CX$ of two cones.  We will regard $\Sigma
X$ as a pointed space, using one of the cone points (which we denote
$N$) as basepoints.  It is easily seen that $\Sigma X$ is simply
connected if and only if $X$ is arcwise connected.  Also it follows
from the Mayer-Vietoris exact sequence that $h_*(\Sigma X,N) \cong
\tilde h_{*-1}(X)$ for any homology theory $h_*$ (e.g.\ singular
homology or oriented bordism).

\begin{lemma}
  A map $f: X \to Y$ of topological spaces induces a homology
  isomorphism if and only if the unreduced suspension $\Sigma f$ is a
  weak equivalence.
\end{lemma}
\begin{proof}
  If $\Sigma f$ is a weak equivalence, then it induces a homology
  isomorphism $(\Sigma f)_*$, but $\tilde H_n(X) = H_{n+1}(\Sigma X)$,
  so $f_*$ is an isomorphism as well.

  For the converse assume $f_*$ is an isomorphism.  Then $\pi_0(f)$ is
  an isomorphism so we can assume that $X$ and $Y$ are arc connected.
  Then $\Sigma X$ and $\Sigma Y$ are simply-connected.  Then Hurewicz'
  theorem implies that $(\Sigma f)$ is a weak equivalence.
\end{proof}
\bigskip

Below we will sketch a proof of the following Hurewicz theorem for
oriented bordism.
\begin{proposition}
  A map $f: X \to Y$ of simply connected spaces is a weak homotopy
  equivalence if and only if it induces an isomorphism in oriented
  bordism.
\end{proposition}
Using this we can easily prove the main theorem
\begin{theorem}
  A map $f: X \to Y$ of topological spaces is a homology isomorphism
  if and only if it induces an isomorphism in oriented bordism.
\end{theorem}
\begin{proof}
  We proved above that $f$ is a homology equivalence if and only if
  $\Sigma f$ is a weak equivalence.  Using the bordism Hurewicz
  theorem we can prove in the same way that $f$ is a bordism
  isomorphism if and only if $\Sigma f$ is a weak equivalence.
\end{proof}

We sketch a proof of a bordism Hurewicz theorem.  Let $h_n:
\pi_n(X,x_0) \to \Omega_n(X,x_0)$ be the map that maps the homotopy
class of a map $(D^n, \partial D^n) \to (X, x_0)$ to the bordism class
of the same map.  The composition of $h_n$ with the cycle map is the
classical Hurewicz homomorphism.

\begin{lemma}
  Let $n \geq 2$.  Let $(X,x_0)$ be a pointed topological space with
  $\pi_k(X,x_0) = 0$ for all $k < n$.  Then
  \begin{equation*}
    h_n: \pi_n(X,x_0) \to \Omega_n(X,x_0)
  \end{equation*}
  is an isomorphism.
\end{lemma}
\begin{proof}
  We construct an inverse.  Let $f:(M,\partial M) \to
  (X,x_0)$ represent an element of $\Omega_n(X, x_0)$.  We can assume
  $M$ is connected.  Choose a CW-structure on $(M,\partial M)$ with
  only one top cell $e: D^n \to M^n$.  Let $M^{[n-1]}$ and $M^{[n-2]}$
  denote the skeleta in the chosen structure.  These fit into a
  cofibration sequence
  \begin{equation*}
    M^{[n-1]}/M^{[n-2]} \to
    M/M^{[n-2]} \to M/M^{[n-1]} \to \Sigma( M^{[n-1]}/M^{[n-2]}).
  \end{equation*}
  The last of these maps is a pointed map $S^n \to \vee^k S^n$, where
  $k$ is the number of $n-1$ cells in $(M,\partial M)$.  It must
  induce the zero map in homology, because otherwise we would have
  $H_n(M,\partial M) = 0$.  Hence it is null-homotopic because $n\geq
  2$.  Applying the functor $[-,(X,x_0)]$, pointed homotopy classes of
  maps to $(X,x_0)$, then gives a isomorphisms
  \begin{equation*}
    \pi_n(X,x_0) \to [M/M^{[n-2]},(X,x_0)] \to [(M,\partial M), (X,x_0)]
  \end{equation*}

  Thus $f$ defines an element of $\pi_n(X,x_0)$ which is easily seen
  to depend only on the cobordism class of $f$.  This defines an inverse.
\end{proof}
Notice how surjectivity of the Hurewicz homomorphism is easier than
injectivity.  Surjectivity uses only connectivity estimates, but
injectivity uses a property of the attaching map of the top
dimensional cell in an oriented manifold.

\begin{theorem}
  Let $n \geq 2$.  Let $(X,x_0)$ be a pointed topological space with
  $\pi_1(X,x_0) = 0$ and $\Omega_k(X,x_0) = 0$ for all $k< n$.  Then
  $\pi_k(X, x_0) = 0$ for all $k < n$.
\end{theorem}
\begin{proof}
  This follows from the previous lemma by induction.
\end{proof}
\begin{proof}[Proof of bordism Hurewicz theorem]
  Assume $f: X \to Y$ is a map of simply connected spaces that is a
  bordism isomorphism.  Let $C(f)$ be the mapping cone and $x \in
  C(f)$ the cone point.  It follows from the long exact sequence in
  oriented bordism that $\Omega_*(C(f), x) = 0$.  Hence by the
  previous theorem (using that $C(f)$ is simply connected) we get that
  $C(f)$ is weakly contractible.  Therefore $H_*(C(f), x) = 0$, so $f$
  is a homology isomorphism and hence a weak equivalence by the
  classical Hurewicz theorem.
\end{proof}
 \nocite{*}

\end{document}